\RequirePackage{etex}
\documentclass[]{imsart}
\startlocaldefs

\usepackage{amssymb,amsmath,amsthm}
\usepackage{hyperref}
\usepackage{cleveref}
\usepackage{enumitem}
\usepackage[T1]{fontenc}
\usepackage{graphicx}
\usepackage{algorithm2e}

\usepackage[]{graphics}
\newcommand\pb[1]{{#1}}

\usepackage{hyperref}
\usepackage{cleveref}
\usepackage{autonum}
\usepackage{mathtools}
\usepackage{thmtools}
\usepackage{booktabs} 

\usepackage{algorithmicx}
\hypersetup{breaklinks=true,
    linkcolor=blue,
    citecolor=magenta,
    colorlinks=true,
pdfborder={0 0 0}}

\usepackage{amsthm}

\numberwithin{equation}{section}
\startlocaldefs

\crefname{theorem}{theorem}{theorems}
\crefname{lemma}{lemma}{lemmas}
\crefname{proposition}{proposition}{propositions}
\crefname{assumption}{assumption}{assumptions}
\crefname{example}{example}{examples}
\crefname{corollary}{corollary}{corollaries}

\declaretheorem[name=Theorem,numberwithin=section]{theorem}
\declaretheorem[name=Proposition,sibling=theorem]{proposition}
\declaretheorem[name=Lemma,sibling=theorem]{lemma}
\declaretheorem[name=Corollary,sibling=theorem]{corollary}
\declaretheorem[name=Assumption,numberwithin=section]{assumption}

\numberwithin{equation}{section}
\numberwithin{theorem}{section}

\usepackage{enumitem}

\newcommand{\RE}{\mathsf{RE}}
\newcommand{\nb}{\,\textsc{nb}(\eps)}
\newcommand{\lsb}{\,\textsc{lsb}(\beta^*)}
\newcommand{\lsbw}{\,\textsc{lsb}}
\newcommand{\X}{{X}}
\newcommand{\design}{\X}
\newcommand{\R}{\mathbf{R}}
\newcommand{\E}{\mathbb{E} }
\newcommand{\eps}{\varepsilon}
\newcommand{\hbeta}{\hat\beta}
\DeclareMathOperator*{\argmin}{argmin}

\DeclareMathOperator*{\supp}{supp}

\DeclareMathOperator*{\Tr}{\mathsf{tr} }

\newcommand{\vg}{\boldsymbol{g} }
\newcommand{\vx}{\boldsymbol{x} }

\newcommand{\Rem}{\text{ Rem}}

\newcommand{\risk}{
    \|\design(\hbeta-\beta^*) \|
}
\endlocaldefs
\begin{document}

\begin{frontmatter}
\title{The noise barrier and the large signal bias of the Lasso and other convex estimators}
\author{Pierre C. Bellec}
\address{
 Department of Statistics \\
 Rutgers, The State University of New Jersey \\
 }

\begin{abstract}
Convex estimators such as the Lasso, the matrix Lasso and the group
Lasso have been studied extensively in the last two decades,
demonstrating great success in both theory and practice.  This paper
introduces two quantities, the noise barrier and the large signal bias,
that provides novel insights on the performance of these convex
regularized estimators.

In sparse linear regression, it is now well understood that the Lasso
achieves fast prediction rates, provided that the correlations of the
design satisfy some Restricted Eigenvalue or Compatibility condition,
and provided that the tuning parameter is at least larger than some
threshold.  Using the two quantities introduced in the
paper, we show that the compatibility condition on the design matrix
is actually unavoidable to achieve fast prediction rates with the
Lasso.  In other words, the $\ell_1$-regularized Lasso must incur a
loss due to the correlations of the design matrix, measured in terms
of the compatibility constant.  This result holds for any design
matrix, any active subset of covariates, and any positive tuning
parameter.
We also characterize sharp phase transitions for the
tuning parameter of the Lasso around a critical threshold dependent on
the sparsity $k$.  If $\lambda$ is equal to or larger than this critical threshold,
the Lasso is minimax over $k$-sparse target vectors.  If $\lambda$ is
equal or smaller than this critical threshold, the Lasso incurs a loss of
order $\sigma\sqrt k$,
even if the target vector has far fewer than $k$ nonzero coefficients.
This sharp phase transition highlights a minimal penalty
phenomenon similar to that observed in model selection with $\ell_0$
regularization in \cite{birge2007minimal}. Several results
extend to convex penalties beyond the Lasso, including the nuclear
norm penalty.
\end{abstract}
\end{frontmatter}


\section{Introduction}

We study the linear regression problem 
\begin{equation}
    y = \design \beta^* + \eps,
    \label{linear-model}
\end{equation}
where one observes $y\in\R^n$, the design matrix $\design\in\R^{n\times p}$ is known 
and $\eps$ is a noise random vector independent of $\X$ 
with $\E[\eps]= 0$.
The prediction error of an estimator $\hbeta$
is given by 
\begin{equation}
    \|\hat y - \X\beta^*\|, \qquad \hat y \coloneqq \X \hbeta
\end{equation}
where $\|\cdot\|$ is the Euclidean norm in $\R^n$.
This paper studies the prediction error of convex regularized estimators,
that is, estimators $\hbeta$ that solve the minimization problem
\begin{equation}
    \label{hbeta}
    \hbeta\in\argmin_{\beta\in\R^p}
    \|\design\beta-y\|^2 + 2 h(\beta),
\end{equation}
where $h:\R^p\rightarrow [0,+\infty)$ a seminorm, i.e., $h$ is non-negative, $h$ satisfies the triangle inequality and $h(a\beta) = |a| h(\beta)$ 
for any $a\in\R, \beta\in\R^p$.

Bias and variance are well-defined for linear estimators of the
form $\hat y = A y$ where $A\in\R^{n\times n}$ is a given matrix.
The bias and variance of the linear estimator $\hat y = A y$ are defined by
\begin{equation}
    b(A) = \| (A-I_{n\times n}) \X \beta^* \|,
    \qquad
    v(A) = \E [\|A\eps\|^2 ]
    \label{eq:bias-variance-linear-A}
\end{equation}
and the squared prediction error of such linear estimator satisfies
\begin{equation}
    \E \|\hat y - \X\beta^* \|^2
    = b(A)^2 + v(A).
\end{equation}
For linear estimators, the two quantities \eqref{eq:bias-variance-linear-A}
characterize the squared prediction error of the linear estimator $A y$
and these quantities can be easily interpreted in terms of the singular values of $A$. The above bias-variance decomposition and the explicit formulae \eqref{eq:bias-variance-linear-A} are particularly insightful to design linear estimators, as well as to understand the role of tuning parameters. 

However, for nonlinear estimator such as \eqref{hbeta},
there is no clear generalization of the bias and the variance.
It is possible to define the bias of $\hbeta$ and its variance as 
$b = \|\X(\E[\hbeta] - \beta^*)\|$ and $v = \E\|\X(\hbeta-\E[\hbeta])\|^2$.
These quantities indeed satisfy $b^2+v=\E \|\design(\hbeta - \beta^*)\|^2$,
but $b$ and $v$ are not interpretable because of the non-linearity of $\hbeta$.
If the penalty $h$ is of the form $\lambda N(\cdot)$ for some norm $N$,
it is not even clear whether the quantities $b$ and $v$ are monotonic with respect to the tuning parameter $\lambda$ or with respect to the noise level.
These quantities appear to be of no help to study the prediction performance of $\hbeta$ or to choose tuning parameters.
The first goal of the present paper is to introduce two quantities, namely the noise barrier and the large signal bias, that clearly describes the behavior 
the prediction error of nonlinear estimators of the form \eqref{hbeta},
that are easily interpretable and that can be used to 
understand the impact of tuning parameters.

For linear estimators such as the Ridge regressor, insights on how to choose tuning parameters can be obtained by balancing the bias and the variance, i.e., the quantities $b(A)$ and $v(A)$ defined above.
To our knowledge, such bias/variance trade-off is not yet understood
for nonlinear estimators of the form \eqref{hbeta} such as the Lasso
in sparse linear regression. A goal of the paper is to fill this gap.

Although our main results are general and apply to estimators \eqref{hbeta}
for any seminorm $h$, we will provide detailed consequences of these general results to the Lasso, that is, the estimator \eqref{hbeta} where the penalty function is
\begin{equation}
    \label{h-lasso}
    h(\cdot) = \sqrt n \lambda \|\cdot\|_1,
\end{equation}
where $\lambda\ge 0$ is a tuning parameter.
The Lasso has been extensively studied in the literature since its introduction in \cite{tibshirani1996regression},
see \cite{knight2000asymptotics,sun2012scaled,bickel2009simultaneous,koltchinskii2011nuclear,koltchinskii2011nuclear,lounici2011oracle}.
These works demonstrate that the Lasso with properly chosen tuning parameter enjoys small prediction and small estimation error, even in
the high dimensional regime where $p\ggg n$.
To highlight the success of the Lasso in the high dimensional regime,
consider first a low-dimensional setting
where $n\ggg p$ and $\X$ has rank $p$. Then the least-squares estimator $\hbeta^{ls}$ satisfies
\begin{equation}
    \sigma \sqrt{p - 1}\le \E \|\X(\hbeta^{ls}-\beta^*)\| \le \sigma \sqrt p,
\end{equation}
for standard normal noise $\eps\sim N(0,\sigma^2 I_{n\times n})$ and deterministic design matrix $\X$.
Here, $p$ is the model size, so that the above display can be rewritten informally as
\begin{equation}
    \E \|\X(\hbeta^{ls}-\beta^*)\| \approx \sigma \sqrt{\text{model size}}.
\end{equation}
In the high-dimensional regime where $p\ggg n$, it is now well understood that the Lasso estimator $\hbeta$ with tuning parameter $\lambda = (1+\gamma)\sigma\sqrt{2\log p}$ for some $\gamma>0$ satisfies
\begin{equation}
    \label{lasso-upper-bound-informal}
    \E \|\X(\hbeta-\beta^*)\| \le C(\X^\top\X/n,T) \; \sigma \lambda \sqrt{|\beta^*|_0},
\end{equation}
where $C(\X^\top\X/n,T)$ is a constant that depends on the correlations of the design matrix and the support $T$ of $\beta^*$.
Examples of such constants $C(\X^\top\X/n, T)$ will be given in the
following subsections.
In the above display, $\|\beta^*\|_0=|T|$
is the number of nonzero coefficients of $\beta^*$.
Even though the ambient dimension $p$ is large compared to the number of observations, the Lasso enjoys a prediction error not larger than the square root of the model size up to logarithmic factors, where now the model size is given by $\|\beta^*\|_0$.
To put it differently, the Lasso operates the following dimension reduction: If the tuning parameter is large enough, then the Lasso acts
as if the ambient dimension $p$ was reduced to $\|\beta^*\|_0$.

There is an extensive literature on bounds of the form
\eqref{lasso-upper-bound-informal} for the Lasso,
see for instance
\cite{tibshirani1996regression,knight2000asymptotics,zhang2008sparsitybias,buhlmann2011statistics,sun2012scaled,bickel2009simultaneous,candes_plan2009lasso,koltchinskii2011nuclear,koltchinskii2011nuclear,lounici2011oracle,bellec2016prediction,bellec2016bounds,bellec2016slope} for a non-exhaustive list.
Despite this extensive literature, some open questions remain on the statistical performance of the Lasso. We detail such questions in the following paragraphs.
They are the main motivation behind the techniques and results of the paper
and behind the introduction of the noise barrier and the large signal bias
defined in \Cref{sec:nb,sec:lsb} below.

\subsubsection*{On the performance of the Lasso with small tuning parameters}

The quantity 
\begin{equation}
    \lambda^* = \sigma\sqrt{2\log p}
    \label{universal-lambda}
\end{equation}
is often referred to as the universal tuning parameter.
Inequalities of the form \eqref{lasso-upper-bound-informal} hold
for tuning parameter $\lambda$ strictly larger than $\lambda^*$
\cite{zhang2008sparsitybias,bickel2009simultaneous,buhlmann2011statistics}.
If the sparsity $k \coloneqq \|\beta\|_0$ satisfies $k\ge 2$, recent results have shown that inequality \eqref{lasso-upper-bound-informal} holds for
tuning parameters slightly larger than 
\begin{equation}
    \label{universal-p/s}
    \sigma\sqrt{2\log(p/k)},
\end{equation}
which leads to better estimation and prediction performance \cite{lecue2015regularization_small_ball_I,bellec2016slope}
than the universal parameter \eqref{universal-lambda}.
An improvement over \cite{lecue2015regularization_small_ball_I,bellec2016slope}
is to obtain the optimal rate $2k\log(p/k)$
with exact multiplicative constant for the Lasso with tuning parameter
slightly larger than \eqref{universal-p/s}; we will provide        
examples of such results at the beginning of     \Cref{sec:lasso-noise-barrier}.

However, little is known about the performance of the Lasso with tuning parameter smaller the thresholds \eqref{universal-lambda} and \eqref{universal-p/s}.
Although it is known from practice that the prediction performance of the Lasso can significantly deteriorate if the tuning parameter is too small, theoretical results to explain this phenomenon are lacking.
A question of particular interest 
to identify the smallest tuning parameter that grants
inequalities of the form
\begin{equation}
    \E \risk
    \lesssim
    \lambda \sqrt{\text{model size}},
\end{equation}
where $\lesssim$ is an inequality up to multiplicative constant, and the model size is given by $\|\beta^*\|_0$.
Also of interest is to quantify how large becomes the risk $\E\risk$
for small tuning parameters.

\subsubsection*{Minimal penalty level for the Lasso}
In the context of model selection \cite{birge2001gaussian},
\cite{birge2007minimal} characterized a minimal penalty.
In our linear regression setting, one of the result of \cite{birge2007minimal}
can be summarized as follows. Let $\mathcal M$ be a set of union of possible models
(each model in $\mathcal M$ represents a set of covariates).
Then there exists a penalty of the form
$$\mathsf{pen} (\beta) = 2 \sigma^2 \|\beta\|_0 A(\|\beta\|_0)$$
for some function $A(\cdot)$ such that
the estimator
$$\hbeta^{MS}\in \argmin_{\beta\in \R^p \text{ with support in }\mathcal M}
\|\X \beta - y\|^2 + \lambda \; \mathsf{pen} (\beta)
$$
enjoys to optimal risk bounds and oracle inequalities
provided that $\lambda=(1+\eta)$ for some arbitrarily small constant $\eta>0$.
On the other hand,
the same optimization problem leads to disastrous results
if $\lambda=(1-\eta)$, cf. the discussion next to equations (28)-(29) in \cite{birge2007minimal}.
To our knowledge, prior literature has not yet characterized 
such sharp transition around a minimal penalty level for $\ell_1$-penalization in sparse linear regression.
We will see in \Cref{sec:lasso-noise-barrier} that the minimal penalty level in $\ell_1$-penalization
for $k$-sparse target vectors
is slightly smaller than \eqref{new_L};
see \Cref{subsec_43_critical}.

\subsubsection*{Necessary conditions on the design matrix for fast rates of convergence}
If the Lasso satisfies inequality \eqref{lasso-upper-bound-informal},
then it achieves a fast rate of convergence in the sense that the
prediction rate corresponds to a parametric rate with $\|\beta^*\|_0$
parameters, up to a logarithmic factor.
All existing results on fast rates of convergence for the Lasso
require some assumption on the design matrix.
Early works on fast rates of convergence of $\ell_1$-regularized procedures \cite{candes2007dantzig,zhang2008sparsitybias,zhang2010nearly} assumed that minimal sparse eigenvalue and maximal sparse eigenvalue of the Gram matrix $\X^\top\X/n$ are bounded away from zero and infinity, respectively.
These conditions were later weakened \cite{bickel2009simultaneous}, showing that a Restricted Eigenvalue (RE) condition on the design matrix is sufficient to grant fast rates of convergence to the Lasso and to the Dantzig selector. The RE condition is closely related to having the minimal sparse eigenvalue of the Gram matrix $\X^\top\X/n$ bounded away from zero \cite[Lemma 2.7]{lecue2014sparse}, but remarkably, the RE condition does not assume that the maximal sparse eigenvalue of the Gram matrix $\X^\top\X/n$ is bounded.
Finally, \cite{buhlmann2011statistics} proposed the compatibility condition, which is weaker than the RE condition.
Variants of the compatibility condition were later proposed in \cite{belloni2014pivotal,dalalyan2017prediction}.
The following definition of the compatibility condition is from \cite{dalalyan2017prediction}. It is slightly different from the original definition of \cite{buhlmann2011statistics}.
Given a subset of covariates  $T\subset \{1,...,p\}$ and a constant $c_0\ge 1$, define the compatibility constant by
\begin{equation}
    \phi(c_0,T) \coloneqq 
    \inf_{u\in\R^p: \|u_{T^c}\|_1 < c_0 \|u_T\|_1}
    \frac{\sqrt{|T|} \|\design u\| }{\sqrt n (\|u_T\|_1 - (1/c_0) \|u_{T^c} \|_1 )},
    \label{def-compatibility}
\end{equation}
where for any $u\in\R^p$ and subset $S\subset\{1,...,p\}$,
the vector $u_S\in\R^p$ is defined by $(u_S)_j = u_j$ if $j\in S$ and $(u_S)_j = 0$ if $j\notin S$.
For $c_0=1$, the constant $\phi(1,T)$ is also considered in \cite{belloni2014pivotal}.
We say that the compatibility condition holds if $\phi(c_0,T)>0$.
If the target vector is supported on $T$, then 
the Lasso estimator $\hbeta$ with tuning parameter $\lambda = (1+\gamma)\sigma\sqrt{2\log p}$ for some $\gamma>0$ satisfies
\begin{equation}
    \E \|\X(\hbeta-\beta^*)\| \le \frac{\lambda}{\phi(c_0,T)} \sqrt{|T|}
    +\sigma\sqrt{|T|} + 4\sigma
    \label{risk-bound-compatbility}
\end{equation}
for $c_0=1+1/\gamma$.
Although we have stated the above display in expectation for brevity,
such results were initially obtained with probability at least $1-\delta$ for the tuning parameter $(1+\gamma)\sigma\sqrt{2\log(p/\delta)}$,
where $\delta$ is a predefined confidence level
\cite{bickel2009simultaneous}, see also the books
\cite{buhlmann2011statistics,giraud2014introduction}.
It is now understood that the tuning parameter $(1+\gamma)\sigma\sqrt{2\log(p)}$ enjoys such prediction guarantee for any confidence level \cite{bellec2016slope}, and that this feature is shared by all convex penalized least-squares estimators \cite{bellec2016bounds}.
Results in expectation such as \eqref{risk-bound-compatbility} are a consequence of the techniques presented in \cite{bellec2016bounds,bellec2016slope}. We refer the reader to Proposition 3.2 in \cite{bellec2016slope}, which implies that for any $\delta\in(0,1)$, inequality
\begin{equation}
    \|\X(\hbeta-\beta^*)\| \le \frac{\lambda}{\phi(c_0,T)} \sqrt{|T|}
    +\sigma\sqrt{|T|} + \sqrt{2\log(1/\delta)} + 2.8\sigma.
\end{equation}
holds with probability at least $1-\delta$. 
Inequality \eqref{risk-bound-compatbility} is then obtained by integration.
Let us mention that results of the form \eqref{risk-bound-compatbility}
are also available in the form of oracle inequalities \cite{bickel2009simultaneous,bellec2016slope}.

In this extensive literature, all results that feature fast rates of convergence for the Lasso require one of the condition mentioned above on the design matrix.
The RE and compatibility conditions are appealing for several reasons.
First, large classes of random matrices are known to satisfy these conditions with high probability, see for instance \cite{lecue2014sparse} for recent results on the "the small ball" method, or \cite[Section 8]{bellec2016slope} \cite{bellec2017towards} for a survey of some existing results.
Second, their introduction has greatly simplified the proofs
that the Lasso achieves fast rates of convergence.
But, to our knowledge, there is no evidence that the above conditions
are necessary to obtain fast rates of convergence for the Lasso.
It is still unknown whether these conditions are necessary,
or whether they are artifacts of the currently available proofs.
The following heuristic argument suggests that a minimal sparse eigenvalue condition is unavoidable, at least to obtain fast rates of estimation.
Given oracle knowledge of the true support of $\beta^*$, one may consider the oracle least-squares estimator on the support $T$ of $\beta^*$, which
has distribution $\hbeta_{oracle}^{ls}\sim N(\beta^*, \sigma^2(\X_T^\top\X_T)^{-1})$ if $\eps\sim N(0,\sigma^2 I_{n\times n})$
and $\X$ is deterministic.
Here, $\X_T$ denotes the restriction of the design matrix to the support $T$.
Then
\begin{equation}
    \E\|\hbeta_{oracle}^{ls} -\beta^*\|^2
    =\frac{\sigma^2}{n}  \sum_{j=1}^{|T|}\frac{1}{\sigma_j},
\end{equation}
where $\sigma_1,...,\sigma_{|T|}$ are the eigenvalues of $\frac 1 n \X_T^\top\X_T$.
Hence, the estimation error diverges as the minimal eigenvalue of $\frac 1 n \X_T^\top\X_T$ goes to 0.
This suggests that, to achieve fast rates of estimation,
the minimal sparse eigenvalue must be bounded away from 0.
A counterargument is that this heuristic only applies to the estimation error, not the prediction error $\risk$.
Also, this heuristic applies to the oracle least-squares but not to the Lasso.

Experiments suggest that the prediction performance of the Lasso deteriorates in the presence of correlations in the design matrix,
but few theoretical results explain this empirical observation.
Notable exceptions include \cite{candes_plan2009lasso}, 
\cite[Section 4]{dalalyan2017prediction} and \cite{zhang2017optimal}: These works exhibit specific design matrices $\X$
for which the Lasso cannot achieve fast rates of convergence for prediction, even if the sparsity is constant.
However, these results only apply to specific design matrices.
One of the goal of the paper is to quantify, for any design matrix $\X$, how the correlations impact the prediction performance of the Lasso.

\subsection*{Organization of the paper}

Let us summarize some important questions raised
in the above introduction.
\begin{enumerate}
    \item How to generalize bias and variance to convex penalized estimators \eqref{hbeta} that are nonlinear?
        If these quantities can be generalized to nonlinear estimators such as
        the Lasso in sparse linear regression,
        how is the choice of tuning parameters related to a
        bias/variance trade-off?
    \item How large is the prediction error of the Lasso when the tuning parameter is smaller than the thresholds \eqref{universal-lambda} and \eqref{universal-p/s}? Does $\ell_1$ regularization present a minimal penalty phenomenon such as
        that observed in $\ell_0$ regularization in \cite{birge2007minimal}?
    \item What are necessary conditions on the design matrix to obtain fast rates of convergence for the Lasso?
        The RE and compatibility conditions are known to be sufficient, but are they necessary?
    \item Is it possible to quantify, for a given design matrix, how the correlations impact the prediction performance of the Lasso?
\end{enumerate}

\Cref{sec:nb,sec:lsb} define two quantities, the noise barrier and the large signal bias, that will be useful to describe the performance
of convex regularized estimators of the form \eqref{hbeta}.
\Cref{sec:compatibility-necessary}
establishes that, due to the large signal bias, the compatibility
condition is necessary to achieve fast prediction rates.
\Cref{sec:lasso-noise-barrier}
studies the performance of the Lasso estimator for tuning
parameters smaller than \eqref{universal-lambda} and
\eqref{universal-p/s}. In particular, \Cref{sec:lasso-noise-barrier} describe a phase transition and a bias/variance trade-off
around a critical tuning parameter slightly smaller than \eqref{universal-p/s}
(see \Cref{subsec_43_critical}).
Finally, \Cref{sec:nuclear} extends some results on the Lasso
to nuclear norm penalization in low-rank matrix estimation.

\section*{Notation}
We use $\|\cdot\|$ to denote the Euclidean norm in $\R^n$ or $\R^p$. The $\ell_1$-norm of a vector is denoted by $\|\cdot\|_1$ 
and the matrix operator norm is denoted by $\|\cdot\|_{op}$.
The number of nonzero coefficients of $\beta\in\R^p$ is denoted by $\|\beta\|_0$.
The square identity matrices of size $n$ and $p$ are denoted by $I_{p\times p}$ and $I_{n\times n}$,
and $(e_1,...,e_p)$ is the canonical basis in $\R^p$.

Throughout the paper, $[p]=\{1,...,p\}$
and $T\subset [p]$ denotes a subset of covariates.
We will often take $T=\{j\in\R^p: \beta_j^*\ne 0\}$
to be the support of $\beta^*$.
If $u\in\R^p$, the vector $u_T$ is the restriction of $u$ to $T$
defined by
$(u_T)_j = u_j$ if $j\in T$ and $(u_T)_j=0$ if $j\notin T$.

Some of our results will be asymptotic.
When describing an asymptotic result, we implicitly
consider a sequence of regression problems indexed by some implicit integer $q\ge0$. The problem parameters and random variables
of the problem (for instance, $(n,p,k,\X,\beta^*,\hbeta)$)
are implicitly indexed by $q$, and we will specify
asymptotic relations between these parameters, see for instance \eqref{extra-asymptotic-regime-k} below.
When such asymptotic regime will be specified,
the notation $a\asymp b$
for deterministic quantities $a,b$ 
means that $a/b\to 1$,
and $o(1)$ denotes a deterministic sequence that converges to 0.

\section{The noise barrier and the large signal bias}

\subsection{The noise barrier}
\label{sec:nb}

Consider the linear model \eqref{linear-model} and let $\hbeta$ be defined by \eqref{hbeta}.
We define the noise barrier of the penalty $h$ by
\begin{equation}
    \nb \coloneqq 
    \sup_{u\in \R^p: \|\X u\|\le 1}(
    \eps^\top \X u
    -
    h(u)
    ).
    \label{def-nb}
\end{equation}
\begin{proposition}
    \label{prop:nb}
    Assume that the penalty $h$ is a seminorm. The noise barrier enjoys the following properties.
    \begin{itemize}
        \item For any realization of the noise $\eps$ we have
            \begin{equation}
                \nb \le \|\X(\hat\beta - \beta^*)\|.
                \label{nb-lower-bound}
            \end{equation}
        \item If $\X\beta^* = 0$ then
            \eqref{nb-lower-bound} holds with equality.
        \item
            If the penalty function $h$ is of the form $h(\cdot) = \lambda N(\cdot)$ for some norm $N(\cdot)$ and $\lambda\ge 0$, then $\nb$ is non-increasing with respect to $\lambda$.
    \end{itemize}
\end{proposition}
\begin{proof}[Proof of \Cref{prop:nb}]
    The optimality conditions of the optimization problem \eqref{hbeta}
    yields that there exists $\hat s$ in the subdifferential of $h$ at $\hbeta$ such that
    \begin{equation}
        \X^\top(\eps + \X(\beta^* - \hbeta) ) = \hat s.
        \label{optimality-condition-hbeta}
    \end{equation}
    The subdifferential of a seminorm at $\hbeta$ is made of vectors $\hat s\in \R^p$ such that $\hat s^\top\hbeta = h(\hbeta)$ and $\hat s^\top u \le h(u)$ for any $u\in\R^p$.
    For any $u\in\R^p$, we thus have
    \begin{equation}
        \eps^\top\X u - h(u) \le u^\top \X^\top \X(\hbeta - \hbeta^*) .
    \end{equation}
    If $\|\X u\|\le 1$, then applying the Cauchy-Schwarz inequality to the right-hand side yields \eqref{nb-lower-bound}.

    If $\X\beta^* = 0$, to prove the second bullet it is enough to prove that $\|\X\hbeta\| \le \nb$.
    Multiplication of \eqref{optimality-condition-hbeta} by $\hbeta$ yields
    $\eps^\top\X\hbeta - \|\X\hbeta\|^2 = h(\hbeta)$.
    If $\X\hbeta \ne 0$ then dividing by $\|\X\hbeta\|$ yields \eqref{nb-lower-bound} with equality.
    If $\X\hbeta = 0$ then $\|\X\hbeta\| \le \nb$ trivially holds.
    Regarding the third bullet, the monotonicity with respect to $\lambda$
    is a direct consequence of definition \eqref{def-nb}.
\end{proof}

The lower bound \eqref{nb-lower-bound}, which holds with probability 1,
is equivalent to
\begin{equation}
    \forall u\in\R^p,
    \qquad
    \eps^\top\X u
    - h(u)
    \le
    \|\X u\| \risk.
    \label{nb-lower-bound-forall-u}
\end{equation}

Intuitively, the noise barrier captures how well the penalty represses the noise vector $\eps$. If the penalty dominates the noise vector uniformly then the noise barrier is small. On the contrary,
a weak penalty function
(in the sense that for some $u\in\R^p$, the quantity $h(u)$ is too small compared to $\eps^\top \X u$) 
will induce a large prediction error because of \eqref{nb-lower-bound}.

The noise barrier for norm-penalized estimators shares similarities with
the variance for linear estimators defined by $v(A)$ in \eqref{eq:bias-variance-linear-A}.
In the absence of signal ($\X\beta^*=0$), for norm-penalized estimators we have $\E[\nb^2] = \E[\risk^2]$, while for linear estimators
$v(A) = \E[\risk^2]$. 
The noise-barrier is non-increasing with respect to $\lambda$ if $h(\cdot)=\lambda N(\cdot)$ for some norm $N$, and 
a similar monotonicity property holds
for linear estimators such as Ridge regressors or cubic splines, given by
\begin{equation}
    \hbeta^{lin} = \argmin_{\beta\in\R^p} \|y-\X\beta\|^2 + 2 \lambda \beta^\top K \beta,
\end{equation}
where $K\in\R^{p\times p}$ is a positive semi-definite penalty matrix.

The noise barrier defined above depends on the noise vector and the penalty function, but not on the target vector $\beta^*$.
The next section defines a quantity that resembles the bias for linear estimators: it depends on $\beta^*$ but not on the noise random vector $\eps$.

\subsection{The large signal bias}
\label{sec:lsb}
 
We will study the linear model \eqref{linear-model}
in the high-dimensional setting where $n$ may be smaller than $p$.
In the high-dimensional setting where $n>p$, 
the design matrix $\X$ is not of rank $p$ and $\beta^*$ is possibly unidentifiable because there may exist multiple $\beta\in\R^p$ such that $\X\beta = \E[y]$. Without loss of generality, we will assume throughout the rest of the paper that $\beta^*$ has minimal $h$-seminorm, in the sense that $\beta^*$ is a solution to the optimization problem
\begin{equation}
    \min_{\beta\in\R^p: \X\beta = \X\beta^*} h(\beta).
    \label{beta-star-minimal-h}
\end{equation}
The large signal bias of a vector $\beta^*\ne 0$ is defined as
\begin{equation}
    \lsb
    \coloneqq
    \sup_{\beta\in \R^p: X\beta\ne X\beta^*} \frac{h(\beta^*) - h(\beta)}{\|\X(\beta^*-\beta)\|}
    \label{def-lsb}
\end{equation}
with the convention $\lsb = 0$ for $\beta^* = 0$.
\begin{proposition}
    \label{prop:lsb}
    Assume that the penalty $h$ is a seminorm. 
    The large signal bias enjoys the following properties.
    \begin{itemize}
        \item For any $\beta^*\ne 0$ we have $h(\beta^*)/\|\X\beta^*\| \le \lsb <+\infty$ provided that $\beta^*$ is solution of the optimization problem \eqref{beta-star-minimal-h}.
        \item
            For any scalar $t\in\R$, $\lsbw(t\beta^*) = \lsb$.
        \item For any small $\gamma>0$,
            if $\X$ is deterministic and if the noise satisfies $\E[\eps] = 0, \; \E[\|\eps\|^2] < +\infty$ and if $\|\X\beta^*\|$ is large enough then 
            \begin{equation}
                (1-\gamma) \; \lsb
                \le 
                \E[ \|\X(\hat\beta-\beta^*)\| ].
                \label{lsb-lower-bound}
            \end{equation}
        \item In the noiseless setting ($\eps = 0$), then
            $\|\X(\hat\beta-\beta^*)\| \le \lsb$ for any $\beta^*\in\R^p$.
        \item
    If the penalty function $h$ is of the form $h(\cdot) = \lambda N(\cdot)$ for some norm $N(\cdot)$ and $\lambda\ge 0$, then $\lsb$ is increasing with respect to $\lambda$.
    \end{itemize}
\end{proposition}

\begin{proof}[Proof of \Cref{prop:lsb}]
    Inequality $h(\beta^*)/\|\X\beta^*\| \le \lsb$ is a direct consequence of the definition \eqref{def-lsb}.
Using Lagrange multipliers, since $\beta^*$ is solution
of the minimization problem \eqref{beta-star-minimal-h},
there is a subdifferential $d$ of $h$ at $\beta^*$ such that $d=\lambda^\top X$ for some $\lambda\in \R^n$.
Then $\lsb$ is bounded from above by $\|\lambda\|$ since
\begin{equation}
    \frac{h(\beta^*) - h(\beta)}{\|\X(\beta^*-\beta)\|}
    \le
    \frac{d^\top(\beta^*-\beta)}{\|\X(\beta^*-\beta)\|} 
    =
    \frac{\lambda^\top X(\beta^*-\beta)}{\|\X(\beta^*-\beta)\|}
    \le \|\lambda\|
    < +\infty
\end{equation}
where we used that $h(\beta)-h(\beta^*)\ge d^\top(\beta-\beta^*)$ by convexity.

    For the second bullet, by homogeneity it is clear that if $\beta\in\R^p$  then $\tilde\beta = t \beta$ satisfies
    \begin{equation}
        \frac{h(\beta^*) - h(\beta)}{\|\X(\beta^*-\beta)\|}
        =
        \frac{h(t \beta^*) - h(\tilde\beta)}{\|\X(t\beta^*- \tilde\beta)\|},
    \end{equation}
    hence $\lsb = \lsbw(t\beta^*)$.
    Inequality \eqref{lsb-lower-bound} is proved below.
    In the noise-free setting ($\eps = 0$), the optimality condition \eqref{optimality-condition-hbeta} yields that 
    $\risk^2 = \hat s^\top(\beta^*-\hbeta) \le h(\beta^*) - h(\hbeta) \le \lsb \risk$. Dividing by $\risk$ shows that $\risk \le \lsb$.
    Regarding the last bullet, the monotonicity with respect to $\lambda$
    is a direct consequence of the definition \eqref{def-lsb}.
\end{proof}

\begin{proof}[Proof of \Cref{lsb-lower-bound}]
    Let $v\in \R^p$ be a nonzero direction, and
    assume that $\beta^*$ is of the form
    $t v$ for some $t>0$. The direction $v$ is fixed throughout,
    and we will show that \eqref{lsb-lower-bound} holds as long
    as $t$ is large enough.

    Let $\hat r=\risk$ for brevity.
Let $\beta\in\R^p$
be deterministic vector that will be specified later.
Multiplication of \eqref{optimality-condition-hbeta} by $\beta-\hbeta$ yields
\begin{equation}
    (\beta-\hbeta)^\top\X^\top(\eps + \X(\beta^*-\hbeta)) = \hat s^\top \beta - \hat s^\top \hbeta \le h(\beta) - h(\hbeta).
\end{equation}
The last inequality is a consequence of $\hat s$ being in the subdifferential of the seminorm $h$ at $\hbeta$, hence $h(\hbeta) = \hat s^\top \hbeta$ and $\hat s^\top\beta \le h(\beta)$.
  The previous display can be rewritten as
  \begin{align}
      \eps^\top\X(\beta-\beta^*) + h(\beta^*) - h(\beta)
      \le
      &(\beta^*-\beta)\X^\top\X(\beta^*-\hbeta) 
      \\
      &+ \eps^\top\X(\hbeta-\beta^*) + h(\beta^*) - h(\hbeta)  - \risk^2 .
      \label{eq:multline}
  \end{align}
  By homogeneity (since $h$ is a seminorm), the quantity
  \begin{equation}
  b =
  \sup_{\beta \in \R^p : X\beta \ne X\beta^*}
  \frac{h(\beta^*) - h(\beta)}{\|\X(\beta^*-\beta)\|}
  =
  \sup_{u\in\R^p : X u \ne X v}
  \frac{h(v) - h(u)}{\|\X(v - u)\|},
  \label{supremum}
  \end{equation}
  only depends on the direction $v$ but not on the amplitude $t>0$
  (from the expression $\beta^* = t v$).

  The second line of the right-hand side of \eqref{eq:multline} is bounded from above by
  $$(\|\eps\| + b ) \risk - \risk^2 \le (1/2) (\|\eps\|^2 + b^2),$$
  thanks to the elementary inequality $(a+a')\risk - \risk^2\le a^2/2 + a'^2/2$.
  The first line of the right-hand side of \eqref{eq:multline} is bounded from above by
  $\|\X(\beta^*-\beta)\| \risk$ by the Cauchy-Schwarz inequality.
  Dividing by $\|\X(\beta^* - \beta)\|$ and taking expectations
  (here, $X(\beta^*-\beta)$ is deterministic and $\E[\eps]=0$)
  we have established that
  \begin{equation}
      \frac{h(\beta^*) - h(\beta)}{\|X(\beta^*-\beta)\|}
      \le
      \E[\risk]
      + \frac{(\E[ \|\eps\|^2] + b^2) }{2\|\X(\beta^*-\beta)\|}.
  \end{equation}
By definition of the supremum in \eqref{supremum},
there exists $u\in\R^p$ such that
$(1-\gamma/2) b \le \frac{h(v) - h(u)}{\|\X(v - u)\|}$ so that
if we set $\beta$ in the above display equal to $tu$ (for the same $t$
that satisfies $\beta^*=tv$), we obtain
$$
(1-\gamma/2) b 
\le
\E[\hat r] + \frac{(\E[ \|\eps\|^2] + b^2) }{2t \|X(v - u)\|}.
$$
For $t$ large enough, namely
$t \ge (\E[\|\eps\|^2] + b^2)/(\gamma b \|X(v-u)\|)$, we obtain
$(1-\gamma) b \le \E[\hat r]$ as desired.
\end{proof}

\section{Compatibility conditions are necessary for fast prediction rates}
\label{sec:compatibility-necessary}
This section explores some consequence of inequality \eqref{lsb-lower-bound}.
Consider the Lasso penalty \eqref{h-lasso},
as well as the compatibility constant defined in \eqref{def-compatibility}
that we recall here for convenience
$$
    \phi(c_0,T) \coloneqq 
    \inf_{u\in\R^p: \|u_{T^c}\|_1 < c_0 \|u_T\|_1}
    \frac{\sqrt{|T|} \|\design u\| }{\sqrt n (\|u_T\|_1 - (1/c_0) \|u_{T^c} \|_1 )}.
$$
The next result shows that the compatibility constant is necessarily bounded from below if the Lasso estimator enjoys fast prediction rates over a given support.

\begin{theorem}
    \label{thm:compatibility-constant-necessary}
    Let $\X$ be a deterministic design matrix,
    let $T\subset [p]$ be any given support
    and assume that $\phi(1,T)>0$.
    Consider the estimator \eqref{hbeta} with penalty \eqref{h-lasso}, for any $\lambda>0$.
    Let $\gamma>0$ be any arbitrarily small constant.
    Assume that the noise satisfies $\E[\eps]=0$ and $\E[\|\eps\|^2]<+\infty$.
    Then there exists $\beta^*\in\R^p$ supported on $T$
    such that
    \begin{equation}
        (1-\gamma) \, \frac{\lambda |T|^{1/2}}{\phi(1,T)} \le \E[\risk].
        \label{lower-compatibility}
    \end{equation}
\end{theorem}
\begin{proof}[Proof of \eqref{lower-compatibility}]
Set $c_0=1$ and assume for simplicity.
By property of the infimum in \eqref{def-compatibility},
for any $\gamma>0$
there exists $u\in\R^p$  such that
\begin{equation}
\sqrt n \lambda \frac{\|u_T\|_1- \|u_{T^c}\|_1}{\|\X u\|}
=
\frac{h(u_T)-h(u_{T^c})}{\|\X u\|}
\ge
(1-\gamma)
\frac{\lambda \sqrt{|T|}}{\phi(1,T)}.
\end{equation}
By homogeneity, we may assume that $\|\X u\|=1$. Next, let $r>0$ be a deterministic
scalar that will be specified later,
set $\beta^*=r u_T$ and $\beta= - r u_{T^c}$ so that, 
for every scalar $r>0$ we have $\beta^*-\beta = r u$ and 
\begin{equation}
    \frac{h(\beta^*)-h(\beta)}{\|\X(\hbeta-\beta^*)\|} 
    =
    \frac{h(u_T)-h(u_{T^c})}{\|\X u\|}.
    \label{fewjioewbao-comp}
\end{equation}
We can take $r$ arbitrarily, in particular large enough so that 
\eqref{lsb-lower-bound} holds.
\end{proof}

An equivalent statement of the previous theorem is as follows:
If the Lasso estimator has a prediction error bounded from above by $C(T)\lambda |T|^{1/2}$ for some constant $C(T)$ uniformly over all target vectors $\beta^*$ supported on $T$, then
\begin{equation}
    \phi(1,T)^{-1} \le C(T).
    \label{ineq:C(T)-compatibility}
\end{equation}
This is formalized in the following.

\begin{theorem}
    Let $\X$ be deterministic, let $T\subset [p]$ and assume that $\phi(1,T)>0$.
    Assume that the noise satisfies $\E[\eps]=0$ and $\E[\|\eps\|^2]<+\infty$.
    If the Lasso estimator \eqref{hbeta} with penalty \eqref{h-lasso} 
    satisfies
    \begin{equation}
        \sup_{\beta^*\in\R^p:\supp(\beta^*)\subset T}
        \E_{\beta^*}[\risk] \le C(T) \lambda |T|^{1/2}
    \end{equation}
    for some constant $C(T)>0$ that may depend on $T$ but is independent of $\beta^*$, then $C(T)$ is bounded from below
    as in \eqref{ineq:C(T)-compatibility}.
\end{theorem}
The above results are a direct consequence of the definition of the large signal bias and inequality \eqref{lsb-lower-bound}.
In the above theorems, the design matrix $\X$
and the support $T$ are not specific: 
The above result applies to \emph{any} $\X$ and \emph{any}
support $T$ such that $\phi(1,T)$ is nonzero.

Known upper bounds on the prediction error of the Lasso include
the so-called ``slow-rate`` upper bound, of the form
$\risk^2 \le 4 \sqrt{n}\lambda \|\beta^*\|_1$ with high probability,
see \cite{sun2012scaled,geer2013lasso,hebiri2013correlations,dalalyan2017prediction} for more precise statements.
In \Cref{thm:compatibility-constant-necessary} above,
the target vector $\beta^*$ has large amplitude, so that
$\|\beta^*\|_1$ is large and
the slow-rate upper bound is not favorable compared to 
the fast-rate bound \eqref{risk-bound-compatbility}.
This also explains that the lower bound \eqref{lower-compatibility}
is not in contradiction with the slow-rate upper bound.

Several conditions have been proposed to provide sufficient
assumptions for fast prediction rates:
the Restricted Isometry property
\cite{candes2005decoding,candes2007dantzig,candes2008restricted},
the Sparse Riesz condition
\cite{zhang2008sparsitybias},
the Restricted Eigenvalue condition\cite{bickel2009simultaneous},
the Compatibility condition 
\cite{buhlmann2011statistics},
the Strong Restricted Convexity condition \cite{negahban2012unified}
and the Compatibility Factor \cite{belloni2014pivotal,dalalyan2017prediction},
to name a few.
The two theorems above are of a different nature. They show that for \emph{any} design matrix $\X$, the Lasso may achieve fast prediction rates only if $\phi(1,T)$ is bounded away from 0. Hence, the compatibility condition with constant $c_0=1$, i.e.,
the fact that $\phi(1,T)$ is bounded from 0,
is necessary to obtain fast prediction rates over the support $T$.

If the diagonal elements of $\frac 1 n \X^\top\X$ are no larger than 1, i.e.,
\begin{equation}
    \label{normalization}
    \max_{j=1,...,p} (1/n) \|\design e_j\| \le 1,
\end{equation}
which corresponds to column normalization,
then the compatibility constant $\phi(1,T)$ is less than 1.
To see this, consider a random vector in $\{-1,1\}^p$ with iid Rademacher coordinates and let $Z$ be the restriction of this random vector to the support $T$ so that $\|Z\|_1 = |T|$.
Then by independence of the coordinates, 
$\E_Z\left[ \frac{|T|\|\X Z\|^2}{\sqrt n \|Z\|_1} \right]\le 1.$
This proves that the compatibility constant $\phi(c,T)$ is less than 1 for any $c>0$, provided that the columns are normalized
as in \eqref{normalization}.
For orthogonal designs or equivalently the Gaussian sequence model, we have $\phi(1,T)=1$.
As one moves away from orthogonal design, the compatibility constant $\phi(1,T)$ decreases away from 1 and the lower bound \eqref{lower-compatibility} becomes larger.
This shows that the performance of the Lasso is worse than that of soft-thesholding in the sequence model, which is of order $\lambda\sqrt{|T|}$.
So there is always a price to pay for correlations in non-orthogonal design matrices compared to orthogonal designs.

Lower bounds similar to \eqref{lower-compatibility} exist in the literature
\cite{bickel2009simultaneous,lounici2011oracle,belloni2013least}, although these results do not yield the same conclusions
as the above theorems.
Namely,
\cite[(B.3)]{bickel2009simultaneous},
\cite[Theorem 7.1]{lounici2011oracle}
and \cite[(A.1)]{belloni2013least}
state that the lower bound 
\begin{equation}
    \label{lower-bound-literature-B3}
    \frac{\lambda\sqrt{\hat s}}{2 \Phi_{\max}}
    \le \risk
    \qquad
    \text{ holds whenever }
    \qquad \lambda > 2\|\eps^\top\X\|_\infty /\sqrt n
\end{equation}
where $\Phi_{\max}$ is a \emph{maximal} sparse eigenvalue and $\hat s$ is the sparsity of the Lasso.
These papers also provide assumptions under which $\hat s$ is of the same order as $|T|$, the support of $\beta^*$, so that results such as \eqref{lower-bound-literature-B3} resemble the above theorems.
However, since $\Phi_{\max}$ is a maximal sparse eigenvalue, it is greater than 1 if the normalization \eqref{normalization} holds with equality. The left-hand side of \eqref{lower-bound-literature-B3} is thus smaller than $\lambda\sqrt{\hat s}$ which is the performance of soft-thresholding in the sequence model.
Furthermore,
as one moves away from orthogonal design, $\Phi_{\max}$ increases and the lower bound
\eqref{lower-bound-literature-B3} becomes weaker.
In contrast, as one moves away from orthogonal design, the
compatibility constant $\phi(1,T)$ decreases and the lower bound
\eqref{lower-compatibility}
becomes larger. Thus inequality \eqref{lower-compatibility} from \Cref{thm:compatibility-constant-necessary} explains the
behavior observed in practice where correlations in the design
matrix deteriorate the performance of the Lasso.
Finally, upper bounds on the risk of Lasso involve the Restricted Eigenvalue constant or
the Compatibility constant, which resemble \emph{minimal} sparse-eigenvalues.
Thus, existing results  that involve the maximal sparse-eigenvalue such as \eqref{lower-bound-literature-B3} do not
match the known upper bounds.

The paper \cite{vandegeer2018tight}, contemporaneous
with the first version of the present work,
also investigates lower bounds
on the Lasso involving the compatibility constant $\phi(1,T)$.
The main result of \cite{vandegeer2018tight} requires strong assumptions
on the design matrix, the noise level and the tuning parameter: $\lambda$ must be large enough,
the maximal entry of the design covariance matrix must be bounded
and satisfy relations involving $\lambda,n$ and $p$.
In contrast, \Cref{thm:compatibility-constant-necessary} above
makes no assumption on $\lambda$ or the design matrix:
\eqref{lower-compatibility} applies to any tuning parameter $\lambda$,
any noise distribution with finite second moment,
any design $\X$ and any support $T$.

If a lower bound holds for some noise distribution then it should also hold
for heavier-tailed distributions,
because intuitively, heavier-tails make the problem harder.
Existing results cited above hold in a high-probability event
of the form $\{\|\X^\top\eps\|_\infty 2/\sqrt n \le \lambda \}$.
This event is of large probability only for large enough $\lambda$ and light-tailed
noise distributions.
In contrast, an appealing feature of the lower bound
\eqref{lower-compatibility} is that it holds for any tuning parameter and any centered noise distribution with finite second moment.

Information-theoretic lower bounds (see, e.g. \cite{rigollet2011exponential} or \cite[Section 7]{bellec2016slope}) involve an exponential number of measures, each
corresponding to a different support $T$ with $|T|=k$. The results above 
involve a \emph{single} measure supported on any given support $T$.
The lower bound \eqref{lower-compatibility} adapts to $T$ 
through the constant $\phi(1,T)$, while minimax lower bounds are not
adaptive to a specific support.

Information-theoretic arguments lead to minimax lower bounds of the order of $$(k\log(p/k))^{1/2}\phi_{\min}(2k),$$
where $\phi_{\min}(2k)$ is a lower sparse eigenvalue of order $2k$ and $k$ is the sparsity
\cite[Section 7]{bellec2016slope}.
Lower sparse eigenvalues become smaller as one moves away from orthogonal design. Hence, minimax lower bounds
do not explain that the prediction performance deteriorates with correlations in the design, while \eqref{lower-compatibility} does.
In the defense of minimax lower bounds: they apply to any estimators.
The lower bounds
derived above are of a different nature and only apply to the Lasso.

\section{
    Noise barrier and phase transitions for the
    tuning parameter of the Lasso}
\label{sec:lasso-noise-barrier}

In this section, we again consider the Lasso, i.e.,
the estimator \eqref{hbeta} with penalty \eqref{h-lasso},
namely
$h(\cdot) = \sqrt n \lambda \|\cdot\|_1.$
Throughout the section, let also $k\in\{1,...,p\}$ denote
an upper bound on the sparsity of $\beta^*$, i.e., $\|\beta^*\|_0 \le k$.
The previous section showed that the Lasso estimator incurs
an unavoidable prediction error of order at least $\lambda\sqrt k$
for some target vector $\beta^*$ with $k$ nonzero coefficients,
provided that the signal is large enough.
In this section, we look at the degradation of performance
as the tuning parameter $\lambda$ becomes small.

\subsection{Fixed design matrix}

Before focusing on the Lasso with design matrix $X$ having correlated columns, consider first the sparse sequence model where one observes
$$z_j = \beta_j^* + g_j,
\qquad
g_j\sim N(0,\sigma^2)
\text{ for each }
j=1,...,p$$
with $T=\{j\in[p]:\beta_j^* \ne 0\}$
the set of nonzero coefficient.
Denote by $\hat b_j = \text{sign}(z_j) (|z_j|-\lambda)_+$
the soft-threhsolding of $z_j$. 
Multiplying the KKT conditions
$\hat b_j - z_j + \lambda \partial|\hat\beta_j| =0$
of the soft-thresholding operator by $\hat b_j - \beta_j^*$
gives
$(\hat b_j - \beta_j^*)^2 =(\hat \beta_j -\beta_j^*)g_j
+ \lambda \beta_j^*\partial |\hat b_j|  - \lambda |\hat b_j|$.
Since $\beta_j^*\partial |\hat b_j| \le |\beta_j^*| = s_j\beta_j^*$
where $s_j = \text{sign}(\beta_j^*)$, this gives 
$(\hat b_j - \beta_j^*)^2 \le (\hat b_j - \beta_j^*)(g_j - s_j \lambda)$
so that
the squared loss of $\hat b_j$ satisfies
\begin{align}
\|\hat b - \beta^*\|^2
&= \sum_{j\in [T]}(\hat b_j - \beta_j^*)^2
+ \sum_{j\in T^c} (|g_j|-\lambda)_+^2
\\&\le 
\sum_{j\in [T]}(g_j - \text{sign}(\beta_j^*)\lambda )^2
+ \sum_{j\in T^c} (|g_j|-\lambda)_+^2.
\label{sequence_model}
\end{align}
This equality and upper bound in the sparse sequence model was already known
 from the early papers on sparsity
\cite{donoho1995adapting,johnstone1994minimax_sparse}.
The inequality
$|\beta_j^* - \hat b_j|^2 \le (g_j - \text{sign}(\beta_j^*)\lambda)^2$
used for each $j\in T$ above
becomes tight as $\beta_j^*\to+\infty$.
Taking expectations on both sides gives
$\E[\|\hat b - \beta^*\|^2]
\le k (\sigma^2+\lambda^2) + (p-k) \E_{Z\sim N(0,1)}[(\sigma|Z| -\lambda)_+^2]$, and the choice $\lambda=\sigma\sqrt{2\log(p/k)}$ is known to provide
the minimax rate, namely $\sigma^2 2k \log(p/k)$ for the squared loss
$\E[\|\hat b - \beta^*\|^2]$.
Explicit known bounds on
$\E_{Z\sim N(0,1)}[(|Z| -\lambda/\sigma)_+^2]$ are stated in
\Cref{appendix:propertiesE|Z|-K}.

For the Lasso with a design matrix $X$ with correlated columns,
previous results such as \cite{lecue2015regularization_small_ball_I,bellec2016slope} established that the Lasso enjoys the minimax rate
$\lambda\sqrt{k}$ for $\lambda$ proportional to $\sqrt{2\log(p/k)}$
with multiplicative constant strictly larger than 1, resulting in
the minimax rate $C \sqrt{2 k\log (p/k)/n}$ for a sub-optimal constant
$C>1$.
The next result provides an improvement over this
to obtain the optimal rate $C \sqrt{2 k\log(p/k)/n}$
with multiplicative constant $C$ arbitrarily close to 1
using a self-contained argument.
We state this result next in the context of a fixed design matrix $\X$, with a
short proof extracted from the more complex \Cref{thm:overwhelming}.
The idea is to obtain upper bounds on the prediction error of the same form
as \eqref{sequence_model}, to mimic as closely as possible the argument
giving the minimax rate in the sequence model.

\begin{proposition}
    \label{prop:simple}
    Let $\xi>0,x\in(0,1)$.
    Let $T=\{j\in[p]:\beta_j^*\ne 0\}$
    and denote $k=|T|$.
    If $\lambda = (1+\xi)\sigma \sqrt{2\log (e p/k)}$, the Lasso
    $\hbeta =\min_{b\in\R^p} \|Xb - y\|^2/2 + \sqrt{n} \lambda \|b\|_1$
    satisfies
    \begin{equation}
        \label{eq:thm_new_simple}
    \max\Bigl\{
    \RE(T,\xi) \|\hbeta - \beta^*\|,
    ~
    \frac{\|X(\hbeta-\beta^*)\|}{\sqrt n}
    \Bigr\}
    \le 
    \frac{\lambda(\sqrt{k} + 0.1)
        + \sigma\sqrt{k}(1+x) 
    }{\sqrt n \min(1,\RE(T,\xi))}
    \end{equation}
    where $\RE(T,\xi) =\inf\bigl\{\frac{\|\X h\|}{\|h\| \sqrt n}, h\in\R^p: \xi \sum_{j\in T^c}|h_j| \le 2 (1+\xi)\sqrt{k} \|h\|\bigr\}$
    in the event
    $$
    \Omega = \Bigl\{
        \sum_{j\in T^c}\Bigl(\frac{|\tilde g_j|}{\sigma}-\Bigl(2\log\frac{e p}{k}\Bigr)^{1/2} \Bigr)_+^2 \le k x^2
        ,
        \|X_T(X_T^\top X_T)^{-1} X_T^\top
        \eps\|
        \le \sigma\sqrt{k} + 0.1\lambda
    \Bigr\}
    $$
    where
    $\tilde g_j = {\eps^T (I_n - X_T(X_T^\top X_T)^{-1}X_T^\top)\X e_j}/{\sqrt n}$.
    If $\eps\sim N(0,\sigma^2 I_n)$ is independent of $\X$,
    $\X$ satisfies $\max_{j\in[p]}\|\X e_j\|_2\le 1$,
    then $\mathbb P(\Omega)\to 1$ as $p/k\to+\infty$
    while $(x,\xi)$ are held fixed.
\end{proposition}
\begin{proof}
    Let $T=\{j\in[p]:\beta_j^*\ne 0\}$.
    Following an idea from \cite{dalalyan2017prediction},
    write $y = \tilde\eps + \X\tilde\beta$
    where $\tilde \eps = (I_n - X_T(X_T^\top X_T)^{-1}X_T^\top)\eps$
    and $\tilde \beta = (X_T^\top X_T)^{-1}X_T^\top y
    = \beta^* + (X_T^\top X_T)^{-1}X_T^\top \eps$
    so that it holds
    $X\tilde\beta = \X \beta^* + X_T (X_T^\top X_T)^{-1}X_T \eps$.
    Then $\tilde\beta$ is the least-squares estimator on the support $T$.
    We will first view $\tilde\beta$ as the ground truth
    and focus on the error vector $h\hbeta-\tilde\beta$.
    Replacing $y$ by $\tilde \eps + \X\tilde\beta$ 
    in the KKT conditions $\X^\top(y-\X\hbeta) = \sqrt n \lambda \partial \|\hbeta\|_1$ and multiplying these KKT conditions by $h$, we obtain
    \begin{equation}
        \label{already_done}
    \|\X h\|^2
    \le \tilde\eps^T\X h
    + \lambda\sqrt n(\|\tilde\beta\|_1 - \|\hbeta\|_1).
    \end{equation}
    Let $\tilde g_j = (\X e_j)^\top \tilde \eps/\sqrt n$.
    Then $\tilde g_j=0$ for $j\in T$.
    Furthermore, for $j\in [T]$ we have
    $|\tilde \beta_j| - |\hbeta_j| \le \tilde s_j (\tilde \beta_j - \hbeta_j)
    = - \tilde s_j h_j$
    where $\tilde s_i\in \{-1,1\}$ is the sign of $\tilde\beta_j$
    (if $\tilde\beta_j=0$, simply pick $\tilde s_j=1$). We obtain
    \begin{align}
    \frac{\|\X h\|^2}{\sqrt n}
    &\le 
    -\lambda \sum_{j\in T}
    \tilde s_j h_j 
    + \sum_{j\in T^c}
    (\tilde g_j  - \lambda_0)|h_j|
    - \xi \lambda_0|h_j|
    \\&\le 
    -\lambda \sum_{j\in T}
    \tilde s_j h_j 
    + \sum_{j\in T^c}
    (|\tilde g_j|  - \lambda_0)_+|h_j|
    - \xi \lambda_0|h_j|
    \\&\le \Bigl(\lambda^2 k + \sum_{j\in T^c} (|\tilde g_j|  - \lambda_0)_+^2\Bigr)^{1/2} \|h\|
    - \xi \lambda_0 \sum_{j\in T^c}|h_j|
    \label{some_previous}
    \end{align}
    where $\lambda_0=\sigma\sqrt{2\log(e p/k)}=\lambda/(1+\xi)$
    and we used the Cauchy-Schwarz inequality for the last line.
    We already see here that the first term in \eqref{some_previous}
    mimics the sparse sequence model upper bound \eqref{sequence_model}.
    In the event $\Omega$, since $x\sigma \le \sigma\le \lambda$
    we have shown that $h$ belongs to the set defining
    the $\RE(T,\xi)$ constant (this is the classical RE constant argument
    from \cite{bickel2009simultaneous}).
    Coming back to \eqref{some_previous}, it holds
    $\|\X h\|^2/\sqrt n
    \le \sqrt k (\lambda^2 + t^2 \sigma^2)^{1/2} (\|\X h\|/\sqrt{n}) \RE(T,\xi)^{-1}$ in which we divide both sides by $\|\X h\|/\sqrt n$
    and use the elementary
    inequality
    $(\lambda^2 + x^2\sigma^2)^{1/2}\le \lambda + \sigma x$.
    Since $\|\X\beta^* - \X\tilde\beta\|
    = \|X_T(X_T^\top X_T)^{-1}X_T^\top \eps\|
    \le \sigma \sqrt k + 0.1 \lambda$ in the event $\Omega$,
    we obtain by the triangle inequality
    \begin{align}
    \|X(\hbeta - \beta^*)\|
    &\le
    \sigma \sqrt k + 0.1\lambda + \|X h\|
    \\&\le 
    \sigma \sqrt k + 0.1 \lambda
    + \sqrt k (\lambda + x \sigma) \RE(T,\xi)^{-1}
    \end{align}
    which proves the upper bound on $\|X(\hbeta-\beta^*)\|/\sqrt n$
    in \eqref{eq:thm_new_simple}.
    For the upper bound on $\|\hbeta - \beta^*\|$,
    \begin{align}
    \|\hbeta - \beta^*\|
    & \le \|\tilde\beta - \beta^*\| + \|\hbeta - \tilde\beta\|
  \\& \le \|\tilde\beta - \beta^*\| + (\|\X h\|/\sqrt n) \RE(T,\xi)^{-1}
    \end{align}
    since both $h$ and $\tilde\beta - \beta^*$ belong to the Restricted
    Eigenvalue cone in the definition of $\RE(T,\xi)$.

    Since $\|X_T(X_T^\top X_T)^{-1} X_T^\top\eps\|^2/\sigma^2$
    has chi-square distribution with degrees-of-freedom
    at most $k$ its expectation is at most $k$,
    and the concentration inequality \cite[Theorem 5.6]{boucheron2013concentration}
    thus gives that $\|X_T(X_T^\top X_T)^{-1} X_T^\top\eps\|/\sigma\le \sqrt k + t$ holds
    with probability at least $1-\exp(-t^2/2)$ which converges to 0 
    as $n\to+\infty$
    if $t = 0.1 \lambda/\sigma$ thanks to
    $\lambda/\sigma\to+\infty$ and $p/k\to+\infty$.
    Since $\tilde g_j$ is centered normal with variance at most $\sigma^2$,
    $$
    \frac{1}{k\sigma^2}
    \E
    \sum_{j\in T^c} (|\tilde g_j|  - \lambda_0)_+^2
\le \frac{p}{k} \E_{Z\sim N(0,1)}\Bigl[\bigl(|Z| - (2 \log\tfrac p k)^{1/2}\bigr)_+^2\Bigr]
    $$
    and the right-hand side converges to 0 as $k/p\to 0$
    thanks to \eqref{bound-1/K^3}.
    By Markov's inequality, this proves that the event $\Omega$ has probability tending to 1 as $p/k\to+\infty$.
\end{proof}

The next result derives a lower bound on the performance of the Lasso
with small tuning parameter. The proof uses
the definition of the noise barrier and the lower bound \eqref{nb-lower-bound}.

\begin{theorem}[Lower bound, deterministic design with minimal sparse eigenvalue]
    \label{thm:lasso-d-log-p-d}
    Let $d$ be a positive integer with $d\le p/5$.
    Let $X$ be deterministic and $\eps\sim N(0,\sigma^2 I_n)$.
    Assume that the design matrix satisfies
    \begin{equation}
        \label{eq:RIP(d)}
        \inf_{u\in\R^p: \|u\|_0\le 2d, u\ne 0} \frac{ \|\X u\|}{\sqrt n \|u\|}
        \ge (1-\delta_{2d})
        \qquad
        \text{ for some }
        \delta_{2d}\in (0,1).
    \end{equation}
    Let $h$ be the penalty function \eqref{h-lasso}.
    If the tuning parameter satisfies
    \begin{equation}
        \lambda \le (\sigma(1-\delta_{2d}) / 8)  \sqrt{\log(p/(5d))},
        \label{eq:lambda-too-small-log-p-5d}
    \end{equation}
    then we have
    \begin{equation}
        \lambda \sqrt d
        \le
        \E \risk 
        .
        \label{eq:lasso-d-log-p-d-lower}
    \end{equation}
\end{theorem}
\begin{proof}[ Proof of \Cref{thm:lasso-d-log-p-d} ]
    Taking expectations in \eqref{nb-lower-bound}, we obtain
    \begin{equation}
        \E \sup_{u\in \R^p: \|\design u\| \le 1} [
            \eps^\top\design u - h(u)
        ] \le  \E \|\design(\hbeta-\beta^*)\|.
    \end{equation}
    Let $\Omega \subset\{-1,0,1\}^p$ be given by Lemma~\ref{lemma:extraction}.
    For any $w\in\Omega$, define $u_w$ as $u_w = (1/\sqrt{dn}) w$.
    Then,
    thanks to the properties of $\Omega$ in Lemma~\ref{lemma:extraction}, $\|\design u_w \|
        =
        \|\design w \| / \sqrt{nd}
        \le 1.$
        Next, notice that $h(u_w) = \lambda \sqrt d$ for all $w\in\Omega$.
    Thus
    \begin{equation}
        \E \sup_{w\in \Omega}
            \eps^\top\design u_w
            - \lambda \sqrt d
        \le
        \E \sup_{u\in V: \|\design u\| \le 1} [
            \eps^\top\design u - h(u)
        ].
            \label{to-be-combined-1}
    \end{equation}
    For any two distinct $w,w'\in\Omega$, by Lemma~\ref{lemma:extraction} we have
    $\E[(\eps^\top\design(u_w - u_w'))^2] \ge \sigma^2 (1-\delta_{2d})^2$.
    By Sudakov's lower bound (see for instance Theorem 13.4 in \cite{boucheron2013concentration})
    we get
    \begin{align}
        \label{eq:lower-sudakov}
        \E\Bigl[\sup_{w\in\Omega}\eps^\top\design u_w\Bigr]
        \ge 
        \frac \sigma 2 (1-\delta_{2d}) \sqrt{\log|\Omega|}
        \ge
        \frac \sigma 4
        (1-\delta_{2d})\sqrt{d\log(p/(5d))}
        .
        \label{to-be-combined-2}
    \end{align}
    Combining \eqref{to-be-combined-1} and the previous display,
    we obtain the desired lower bound if
    $\lambda$ satisfies \eqref{eq:lambda-too-small-log-p-5d}.
\end{proof}

\Cref{thm:lasso-d-log-p-d} makes no sparsity assumption on the
target vector $\beta^*$ and its implications are better understood
for target vectors such that $\|\beta^*\|_0 \lll d$.
Even though the size of the true model is sparse and of size $\|\beta^*\|_0 \lll d$,
the Lasso with small tuning parameter (as in \eqref{eq:lambda-too-small-log-p-5d})
suffers a prediction error of order at least $\sqrt{d\log(p/d)}$
and cannot satisfy the dimensionality reduction property of order $\|\beta^*\|_0$.

However, if $k=d$ there is a gap between the right-hand side
of \eqref{eq:lambda-too-small-log-p-5d} (which is smaller than $(\sigma/8)\sqrt{\log(p/(5d))}$) and the tuning parameter
used in \Cref{prop:simple} (larger than $\sigma\sqrt{2\log(ep/k)}$).
The tools used in the proofs of these two results for fixed design do not allow
us to bridge this gap in the leading multiplicative constant.
However, we will bridge this gap in the next subsection for random
Gaussian designs.

\subsection{Random design}

For tuning parameters slightly larger than $\sigma\sqrt{2\log(p/k)}$,
of the form $(1+\xi)\sigma\sqrt{2\log(p/k)}$ for an arbitrary
small constant $\xi$,
a technique similar to \Cref{prop:simple} can be developed for random Gaussian designs.
Again, the strategy in the proof is 
to mimic the sequence model bound \eqref{sequence_model},
and using Gordon's escape through a mesh result \cite{gordon1988milman} to control
restricted eigenvalues of the design matrix.
This is done in \Cref{thm:overwhelming},
which implies the following corollary.

\begin{corollary}[Corollary of \Cref{thm:overwhelming}]
    \label{cor:random}
    Let $\xi>0,x\in(0,1)$.
    Let $T=\{j\in[p]:\beta_j^*\ne 0\}$
    and denote $k=|T|$.
    Let $X\in\R^{n\times p}$ be a random matrix with i.i.d. 
    $N(0,\Sigma)$ rows for some positive definite $\Sigma\in\R^{p\times p}$
    and $C_{\min} = \phi_{\min}(\Sigma)^{1/2}$.
    Let $\eps\in\R^n$ be independent of $X$ with $\|\eps\|^2/n = \sigma^2$
    and $\max_{j\in[p]}\Sigma_{jj}\le 1$.
    Assume
    \begin{equation}
        \label{conditoin_simple_random_design}
        k/p\to 0.
    \end{equation}
    If $\lambda = (1+\xi)\sigma \sqrt{2\log (e p/k)}$, the Lasso
    $\hbeta =\min_{b\in\R^p} \|Xb - y\|^2/2 + \sqrt{n} \lambda \|b\|_1$
    satisfies
    \begin{equation}
        \mathbb P\Bigl(
        \label{eq:thm_random_simple}
        \|\Sigma^{1/2}(\hbeta - \beta^*)\|
        \le \frac{
        \lambda(\sqrt k + 0.1)
        + \sigma \sqrt k ( 1+x)
        }{C_{\min} \sqrt n (1-\bar\rho_n)_+^2}
        \Bigr)\to 1
        \label{eq:random_cor_simple}
    \end{equation}
    where $\bar\rho_n= \frac{C(\xi,x)}{C_{\min}^2} \frac{k}{n}\log \frac{ep}{k}$ 
    and $C(\xi,x)>0$ is a constant depending only on $(\xi,x)$.
\end{corollary}
\Cref{cor:random} is proved in \Cref{appendix:overwhelming-proba}
as a consequence of the more complex \Cref{thm:overwhelming}.
Note that the bound is vacuous unless $\bar\rho_n<1$ due to the
presence of $(1-\bar\rho_n)_+$ in the denominator of the right-hand side.
If $\Sigma=I_p$ as well as $k\to+\infty$ and 
$\frac kp + \frac kn\log(p/k) \to 0$,
then $C_{\min}=1$ and the right-hand side inside the probability
bound is $\sigma (1+\xi) \sqrt{2k\log(p/k)/n} (1+o(1))$.
This recovers the minimax rate in \cite{su2016slope} achieved by Slope
since $\xi>0$ can be taken as an arbitrary small constant. 
Since \cite{su2016slope} showed that Slope is minimax optimal for the
sparse recovery problem under isotropic Gaussian design, alternative
techniques have been proposed that also achieve the minimax rate.
\Cref{prop:simple} and \Cref{cor:random} (which are simplified version
of \Cref{thm:overwhelming}) show that the Lasso also achieves
the optimal rate with $\lambda=(1+\xi)\sigma\sqrt{2\log(p/k)}$.
Section 3.2 and Theorem 4.1 in 
\cite{bellec2019first} showed that the Lasso can be compared in L2 norm,
up to an error term negligible compared to the minimax rate, to the 
denoiser $\argmin_{b\in \R^p}\|\Sigma^{1/2}(b-\beta^* - g)\|^2
+ 2 \lambda \|b\|_1$ where $g_j=\eps^\top \X e_j/\sqrt n$.
The work \cite{ndaoud2020scaled} develops an iterative hard-threhsolding
procedure to achieve the minimax rate. Recently,
\cite{guo2024note} showed that the minimax rate for
isotropic random designs is the same $\sigma^2 2 k\log(p/k)/n$
for the expected squared error $\E[\|\hbeta-\beta^*\|^2]$,
by controlling the expected value on events of vanishing probability
that are typically excluded from high-probability analysis.

At this point, it is still unclear what happens for tuning parameters
$\lambda$ equal or smaller than $\sigma\sqrt{2\log(p/k)}$
as both \Cref{cor:random} and \Cref{prop:simple}
require tuning parameters of the form
$(1+\xi)\sigma\sqrt{2\log(p/k)}$.
The next few results will study this situation, and show that due to
the lower bound \eqref{nb-lower-bound}, tuning parameters slightly smaller
than $\sigma\sqrt{2\log(p/k)}$ suffer degraded performance--much worse
that the minimax rate.

\subsection{Critical tuning parameter: the $5\log\log(p/k)$ correction}
\label{subsec_43_critical}

For a fixed constant $\xi \in (0,1)$,
target sparsity level $k$ and dimension $p$, define the critical
tuning parameter
as
\begin{equation}
    \label{new_L}
L(\tfrac{p}{k}, \xi)
= \sigma\sqrt{
    2\log(\tfrac{p}{k} \tfrac{4}{\xi \sqrt{2 \pi}})
    -
    5\log\log(\tfrac{p}{k})
}.
\end{equation}
We will see next that this tuning parameter
achieves, for $k$-sparse vector $\beta^*$ a performance arbitrarily close to
the minimax risk $\sigma\sqrt{2k\log(p/k)}$
by tweaking the constant $\xi$;
while it incurs a prediction error of the same order
$\sigma\sqrt{2k\log(p/k)}$
for any $\beta^*$ even if $\|\beta^*\|\lll k$.
The key for this property is the $5\log\log$ correction.

\begin{assumption}
    \label{assum:random-design}
    $\xi\in (0,1)$ is a constant as $n,p\to+\infty$.
    The noise vector $\eps$ is deterministic with $\sigma^2 = \|\eps\|^2/n$
        and the design  $\X$ has iid rows distributed as $N(0,\Sigma)$ with $\max_{j=1,...,p}\Sigma_{jj}\le 1$.
\end{assumption}

\begin{theorem}
    \label{thm:critical-upper-lower}
    Let \Cref{assum:random-design} be fulfilled.
    Let $k,n,p$ be positive integers and consider an asymptotic regime
    with
    \begin{equation}
        k,n,p\to+\infty, \qquad
        k/p\to0,\qquad
        k \log^3(p/k)/n\to 0
        \label{extra-asymptotic-regime-k}.
    \end{equation}
    Let $h$ be the penalty \eqref{h-lasso}
    and let $\hbeta$ be the Lasso \eqref{hbeta}
    with tuning parameter $\lambda=L(\tfrac{p}{k}, \xi)$
    in \eqref{new_L}.
    \begin{enumerate}[label=(\roman*)]
        \item             Assume that for some constant $C_{\min}>0$ independent of $k,n,p$ we have $\|\Sigma^{-1}\|_{op}\le 1/C_{\min}^2$.
            Then, if $\|\beta^*\|_0 \le k$ we have
            \begin{equation}
                \mathbb P\left(
                    \frac{\risk}{\sqrt n}
                    \vee
                    \|\Sigma^{1/2}(\hbeta-\beta^*)\|
                    \le
                    [1+o(1)]
                    \frac{
                        \lambda \sqrt k
                        (1+\sqrt\xi)
                    }{C_{\min} \sqrt n}
                    \right) \to 1.
                    \label{critical:upper-bound}
            \end{equation}
        \item
            Assume that $\Sigma_{jj}=1$ for each $j=1,...,p$,
            and assume that $\|\Sigma\|_{op} \le C_{\max}^2$
            for some constant $C_{\max}$ independent of $k,n,p$.
            Then for any $\beta^*\in\R^p$ we have
            \begin{equation}
                \mathbb P\left(
                    [1-o(1)]
                    \frac{
                        \lambda \sqrt k \sqrt \xi
                    }{C_{\max} \sqrt n}
                \le
                    \frac{\risk}{\sqrt n}
                \right) \ge 1/3
                .
                \label{critical:lower-bound}
            \end{equation}
    \end{enumerate}
\end{theorem}
The upper bound \eqref{critical:upper-bound} is proved in
\Cref{appendix:proofs-lasso-lower} and the lower bound
\eqref{critical:lower-bound} in
\Cref{appendix:proofs-lasso-lower}.
In \eqref{critical:upper-bound} and \eqref{critical:lower-bound}, the quantity $o(1)$ is a deterministic positive sequence that converges to 0 in the asymptotic regime \eqref{extra-asymptotic-regime-k}.
\Cref{prop:simple}, \Cref{cor:random}
and \eqref{critical:upper-bound} show that the Lasso with
tuning parameter
equal or larger than \eqref{new_L}  achieves a prediction error
not larger than $\lambda\sqrt k$ if $\beta^*$ is $k$-sparse.
Since $\lambda$ is of logarithmic order, the Lasso thus enjoys
a dimensionality reduction property of order $k$:
Its prediction error is of the same order as that of the Least-Squares
estimator of a design matrix with only $k=\|\beta^*\|_0$ covariates.
The lower bound \eqref{critical:lower-bound} answers the dual question:
Given an integer $k\ge1$,
for which values of the tuning parameter does the Lasso
lack the dimensionality
reduction property of order $k$?

The lower bound \eqref{critical:lower-bound} above holds even for
target vectors $\beta^*$ with $\|\beta^*\|_0 \llless k$.
Let us emphasize that the same tuning parameter \eqref{new_L}
enjoys both the upper bound \eqref{critical:upper-bound} and the lower bound \eqref{critical:lower-bound}.
Furthermore, the right-hand side of \eqref{critical:upper-bound}
and the left-hand side of \eqref{critical:lower-bound} are both
of order $\sigma \sqrt{2k \log(p/k)}\asymp \lambda\sqrt k$.
Hence, in the  asymptotic regime \eqref{extra-asymptotic-regime-k},
the tuning parameter \eqref{new_L} satisfies simultaneously
the two following properties:
\begin{itemize}
    \item By \eqref{critical:upper-bound}, the Lasso with tuning parameter $\lambda=\eqref{new_L}$
        achieves the minimax rate over the class of $k$-sparse target vectors
        $\sigma\sqrt{2k\log(p/k)/n}$, up to a multiplicative constant
        $1+\sqrt\xi$ that can be taken arbitrarily close to 1
        by reducing the constant $\xi>0$.
    \item By \eqref{critical:lower-bound}, the Lasso with tuning parameter $\lambda=\eqref{new_L}$
        suffers a prediction error of order at least $\sigma\sqrt{k\log(p/k)/n}$,
        even if the sparsity of the target vector is negligible compared to $k$.
        This means that $\lambda=\eqref{new_L}$ is too small
            for target vectors with $\|\beta^*\|_0\lll k$ since it will incur a prediction error of order $\sqrt{k\log(p/k)/n}$,
            although one would hope in this case to achieve an error of order $\sqrt{\|\beta^*\|_0 \log(p/\|\beta^*\|_0)/n}$.
        \item The tuning parameter \pb{\eqref{new_L} is too large}
        for target vectors with $\|\beta^*\|_0\ggg k$ since it will incur a prediction error of order $\sqrt{\|\beta^*\|_0\log(p/k)}$, although one would hope 
        in this case to achieve an error of order $\sqrt{\|\beta^*\|_0 \log\frac{p}{\|\beta^*\|_0}}$.
\end{itemize}

The rate $\sigma\sqrt{2k\log(p/k)}$ 
is known to be asymptotically minimax (with exact multiplicative constant)
under \Cref{assum:random-design} with $\Sigma=I_{p\times p}$ and $\theta=1$, that is, 
if $\X$ has iid standard normal entries \cite[Theorem 5.4]{su2016slope}.
\Cref{critical:upper-bound} thus shows that
the Lasso with $\lambda$ given either by \Cref{cor:random}
or by \eqref{new_L}
is minimax over $k$-sparse target vectors, with
exact asymptotic constant. In \Cref{cor:random}
this requires $k\log(p/k)/n + k/p\to 0$, while in \eqref{critical:upper-bound}
for $\lambda=\eqref{new_L}$
we require the slightly stronger condition
\eqref{extra-asymptotic-regime-k}.

Existing results 
require tuning parameters of the form
$(1+\xi)\sigma\sqrt{2\log p}$ for some constant $\xi>0$ \cite{bickel2009simultaneous,candes_plan2009lasso,buhlmann2011statistics,negahban2012unified,dalalyan2017prediction},
or of the form $(1+\xi)\sigma\sqrt{2\log(p/k)}$ where again $\xi>0$ is constant
\cite{sun2012scaled,lecue2015regularization_small_ball_I,bellec2016slope,feng2017sorted}; the constant $\xi>0$ can be made taken arbitrarily close to
1 as in \Cref{prop:simple} and \Cref{cor:random}.
The result above shows that one can use tuning parameters even
smaller than $\sigma\sqrt{2\log(p/k)}$, up to the $5\log\log$
correction visible in \eqref{new_L}.
The $5\log\log$ correction in \eqref{new_L} is key in order to prove,
for the same $\lambda$, both the upper bound \eqref{critical:upper-bound}
and the lower bound \eqref{critical:lower-bound}.

This phenomenon is similar the minimal penalty in $\ell_0$-regularization studied in \cite{birge2007minimal}:
For tuning parameters slightly smaller than \eqref{new_L},
the Lasso will lead to terrible results in the sense that
its prediction error will be much larger
than $\sigma\sqrt{2k\log(p/k)}$.
The phase transition is remarkably sharp:
if $\lambda = L(\tfrac{p}{k\log(p/k)}, \xi)$
then $\|X(\hbeta-\beta^*)\|\ggg \sigma\sqrt{k\log(p/k)}$
with constant probability, even though this
$\lambda$ and \eqref{new_L} only differ by a $\log\log$ term
inside the square root in \eqref{new_L}.

A lower bound comparable to \eqref{critical:lower-bound} was obtained in \cite[Proposition 14]{sun2013sparse} in a random design setting.
\Cref{critical:lower-bound} is different from this result of \cite{sun2013sparse} in at least two ways. First the assumption on $\|\Sigma\|_{op}$ is allowed to be larger in \eqref{critical:lower-bound}.
Second, and more importantly, \Cref{critical:lower-bound} is sharper in the sense that it applies to the critical tuning parameter \eqref{new_L} that achieves the upper bound \Cref{critical:upper-bound},
while the lower bound in \cite{sun2013sparse} only applies to tuning parameters strictly smaller than \eqref{new_L}.

\section{Extension to penalties different from the $\ell_1$ norm}
\label{sec:nuclear}
In this section, we consider penalty functions
different from the $\ell_1$ norm.
We will see that the lower bounds induced by the noise barrier
and the large signal bias are also insightful for these penalty functions and that several of our results on the Lasso can be extended.

%

\subsection{Nuclear norm penalty}
Let $p=mT$ for integers $m\ge T >0$.
We identify $\R^p$ with the space of matrices with $m$ rows and $T$ columns. Let $\beta'$ be the transpose of a matrix $\beta\in\R^{m\times T}$ and denote by $\Tr(\beta_1'\beta_2)$ the scalar product of two matrices
$\beta_1,\beta_2\in\R^{m\times T}$.
Assume that we observe pairs $(X_i,y_i)_{i=1,...,n}$
where $X_i\in\R^{m\times T}$ and 
\begin{equation}
    y_i = \Tr(X_i'\beta^*) + \eps_i, \qquad\qquad i=1,...,n,
    \label{matrix-model}
\end{equation}
where each $\eps_i$ is a scalar noise random variable 
and $\beta^*\in\R^{m\times T}$ is an unknown target matrix.
The model \eqref{matrix-model} is sometimes referred to as the
trace regression model.
Define $y=(y_1,...,y_n)$, $\eps=(\eps_1,...,\eps_n)$ and define
the linear operator $\X:\R^{m\times T}\to\R^n$ by
$(\X\beta)_i = \Tr(X_i'\beta)$ for any matrix $\beta$ and any $i=1,...,n$, so that the model \eqref{matrix-model} can be rewritten
as the linear model
\begin{equation}
    y = \X\beta^* + \eps.
\end{equation}
Define the nuclear norm penalty by
\begin{equation}
    \label{h-nuclear}
    h(\cdot) = \sqrt n \lambda \|\cdot\|_{S_1},
\end{equation}
where $\lambda\ge 0$
and $\|\cdot\|_{S_1}$ is the nuclear norm in $\R^{m\times T}$.
Define also the Frobenius norm
of a matrix $\beta$ by $\|\beta\|_{S_2} = \Tr(\beta'\beta)^{1/2}$.

\begin{theorem}
    \label{thm:nuclear-small-lambda}
    Assume that $\eps\sim N(0,\sigma^2I_{n\times n})$
    and let $r\in \{1,...,T\}$ be an integer such that
    for some $\delta_r\in(0,1)$ we have
    \begin{equation}
        (1-\delta_r)\|\beta\|_{S_2} \le
        \|\X\beta\|/\sqrt n
        \le
        (1+\delta_r)\|\beta\|_{S_2},
        \qquad
        \forall\beta\in\R^{m\times T}: \text{rank }(\beta)\le 2r
    \end{equation}
    Let $h$ be the penalty function \eqref{h-nuclear}.
    Then the nuclear norm penalized least-squares estimator
    \eqref{hbeta} satisfies
    \begin{equation}
        \sqrt r
        \left(
        c \sigma (1-\delta_r)\sqrt m
        -
        \lambda
        \right)
        \le (1+\delta_r)\;\E\risk.
    \end{equation}
    for some absolute constant $c>0$.
\end{theorem}
The proof is given in \Cref{appendix:proof-lower-bound-nuclear-norm}.
If the tuning parameter is too small in the sense
that $\lambda \le c\sigma (1-\delta_r) \sqrt m /2$,
then the prediction error $\E\risk$ is bounded from below by
$\frac{c\sigma(1-\delta_r)}{2(1+\delta_r)} \sqrt{r m}$.
For common random operators $\X$, the Restricted Isometry
condition of \Cref{thm:nuclear-small-lambda} is granted
with $\delta_r=1/2$ for any $r\le T/C$ where $C>0$ is some absolute constant,
see for instance \cite{candes2011tight} and the references therein.
In this case, the above result with $r=T/C$ yields 
that for some absolute constant $c'>0$,
\begin{equation}
    \lambda \le (c/4) \sigma \sqrt{m} 
    \qquad
    \text{ implies }
    \qquad
    c' \sigma \sqrt{T m}
    \le
    \E\risk.
\end{equation}
For tuning parameters smaller than $c\sigma \sqrt m/4$ in expectation,
the performance of the nuclear penalized least-squares estimator
is no better than the performance of the unpenalized least-squares estimator for which $\E\|\X(\hbeta^{ls} - \beta^*)\|$ is of order $\sigma\sqrt p =\sigma\sqrt{mT}$.

The large signal bias lower bound of \eqref{lsb-lower-bound}
yields that for any target vector $\beta^*$ such that $\|\X\beta^*\|$
is large enough and under the mild assumptions of
\eqref{lsb-lower-bound},
we have
\begin{equation}
    0.99 \; \lambda
    \sup_{\beta\in\R^{m\times T}: \X\beta^*\ne \X\beta}
    \frac{\|\beta^*\|_{S_1} - \|\beta\|_{S_1} }{\|\X(\beta^*-\beta)\|/\sqrt n}
    \le
    \E\risk.
\end{equation}
As in \Cref{thm:compatibility-constant-necessary} for the Lasso, the previous display shows
that the nuclear norm penalized estimator will incur a prediction error due to 
correlations in the design matrix $\X$.
\bibliographystyle{imsart-number}
\bibliography{../../bibliography/db}

\appendix

\section{Properties of $\E[(|Z|-K)_+^2]$}
\label{appendix:propertiesE|Z|-K}
Define the pdf and survival function of the standard normal distribution by
\begin{equation}
    \varphi(u)=\frac{e^{-u^2/2} }{\sqrt{2\pi} },
    \qquad\qquad
    Q(\mu) = \int_\mu^{+\infty} \varphi(u)du
    \label{def-varpih-Phi}
\end{equation}
for any $u,\mu\in\R$.

The following bound is known: For any $K>0$ and $Z\sim N(0,1)$, we have
\begin{equation}
    \label{bound-1/K^3}
    \varphi(K)
    \left[\frac{4}{K^3} - \frac{24}{K^5} - \frac{30}{K^7} \right]
    \le
\E[(|Z|-K)_+^2]
\le \frac{4 \varphi(K)}{(K^2+2)(K^2 + 4/\sqrt{2\pi})^{1/2}}
\end{equation}
so that $\E[(|Z|-K)_+^2] K^3/\varphi(K) \to 4$ as $K\to+\infty$.
\begin{proof}[Proof of \eqref{bound-1/K^3}]
    The upper bound is given in \cite{sun2013sparse} 
    or \cite[Appendix]{bellec_zhang2018second_stein}.
    For the lower bound, an integration by parts reveals
    \begin{equation}
        \E[(|Z| - K)_+^2] =  2 \left[ - K\varphi(K) + (K^2+1) Q(K)\right]
        \label{eq:computation}
    \end{equation}
    where $\varphi,Q$ are defined in \eqref{def-varpih-Phi}
    (this formula for $\E[(|Z| - K)_+^2]$
    is obtained for instance in \cite[equation (68)]{chandrasekaran2012convex}
    or \cite[Appendix 11]{johnstone1994minimax_sparse}).
    By differentiating, one can readily verify that
    \begin{equation}
        Q(K) = \varphi(K)\left[\frac 1 K- \frac 1 {K^3} + \frac 3 {K^5} - \frac{15}{K^7}\right] + \int_K^{+\infty} \frac{105 \varphi(u)}{u^8} du.
    \end{equation}
    Since the rightmost integral is positive, this gives a lower bound on $Q(K)$
    which yields the lower bound in \eqref{bound-1/K^3}.

\end{proof}

\section{Signed Varshamov-Gilbert extraction Lemma}

\begin{lemma}[Lemma 2.5 in \cite{giraud2014introduction}]
    \label{lemma:extraction-giraud}
    For any positive integer $d\le p/5$, there exists a subset $\Omega_0$ of the set
    $\{w \in\{0,1\}^p: |w|_0 = d \}$ that fulfills 
    \begin{align}
        \log (|\Omega_0|) &\ge  (d/2)  \log\left(\frac{p}{5d}\right),
        \qquad
        \sum_{j=1}^p \mathbf 1_{w_j \ne w_j'} = \|w - w'\|^2 > d, 
        \label{eq:properties-Omega-0}
    \end{align}
    for any two distinct elements $w$ and $w'$ of $\Omega_0$,
    where $|\Omega_0|$ denotes the cardinality of $\Omega_0$.
\end{lemma}

\begin{lemma}[Implicitly in \cite{verzelen2012minimax}
    and in \cite{bellec2016slope}. The proof below is provided for completeness.]
    \label{lemma:extraction}
    Let $\X\in\R^{n\times p}$ be any matrix with real entries.
    For any positive integer $d\le p/5$, there exists a subset $\Omega$ of the set
    $\{w \in\{-1,0,1\}^p: |w|_0 = d \}$ that fulfills both
    \eqref{eq:properties-Omega-0} and
    \begin{equation}
        \frac 1 n \|\design w\|^2 \le d \;\;
        \max_{j=1,...,p}\frac{\|\X e_j\|^2}{n} 
        ,
        \qquad 
    \end{equation}
    for any two distinct elements $w$ and $w'$ of $\Omega$.
\end{lemma}
\begin{proof}
    By scaling, without loss of generality we may assume that
    the columns are normalized with
    $\max_{j=1,...,p}\|\X e_j\|^2 = n$.
    Let $\Omega_0$ be given from \eqref{lemma:extraction-giraud}.
    For any $u\in \Omega_0$, define $w_u\in \{-1,0,1\}^d$ by
    \begin{equation}
        w_u = \argmin_{v\in\{-1,0,1\}^d:\, |v_j|=u_j\;\forall j=1,...,p} \|\X v\|^2
        \label{eq:optimi-extraction-per-w}
    \end{equation}
    and breaking ties arbitrarily.
    Next, define $\Omega$ by $\Omega \coloneqq \{w_u, u\in\Omega_0\}$.
    It remains to show that 
    $\frac 1 n \|\design w_u\|^2 \le d$.
    Consider iid Rademacher random variables $r_1,...,r_p$ and
    define the random vector $T(u)\in\R^p$ by $T(u)_j = r_j u_j$ for each $j=1,...,p$.
    Taking expectation with respect to $r_1,...,r_p$ yields
    \begin{equation}
        \frac 1 n\E[\|\design T(u)\|^2 ] 
        = \frac 1 n \sum_{j=1}^p |u_j| \|\X e_j\|^2
        \le d.
    \end{equation}
    By definition of $w_u$, the quantity $(1/n)\|\design w_u\|^2$ is bounded from above by the previous display and the proof is complete.
\end{proof}

\section{RIP property of sparse vectors}

\begin{proposition}
    \label{prop:cone-sparse-vectors-d}
    There exists an absolute constant $C>0$ such that the following
    holds.
    Let \Cref{assum:random-design} be fulfilled.
    Then
    \begin{equation}
        \E\left[\sup_{u\in \R^p: \|u\|_0\le d, \|\Sigma^{1/2} u\| =1 }
        \Big| \frac 1 {n} \|\X u\|^2 - \|\Sigma^{1/2}u\|^2\Big|
        \right] 
        \le
        C
        \frac{\sqrt{d\log(e p/d)}}{\sqrt n}
    \end{equation}
    holds and with probability at least $1-e^{-t^2/2}$ we have
    \begin{equation}
        \sup_{u\in \R^p: \|u\|_0\le d, \|\Sigma^{1/2} u\| =1 }
        \Big| \frac 1 {\sqrt n} \|\X u\| - \|\Sigma^{1/2}u\|\Big|
        \le
        C
        \frac{\sqrt{d\log(e p/d)} + t}{\sqrt n}
        .
    \end{equation}
\end{proposition}
\begin{proof}
    We apply the second part of Theorem D in
    \cite{mendelson2007reconstruction} which yields that 
    the left-hand side of the first claim is bounded from above
    by
    $C \Gamma / \sqrt n$
    for some absolute constant $C_1>0$
    where
    $\Gamma=\E\sup_{u\in \R^p: \|u\|_0\le d, \|\Sigma^{1/2} u\|= 1} | G^\top\Sigma^{1/2} u |$
    where the expectation is with respect to $G\sim N(0,I_{p\times p})$.
    By the Cauchy-Schwarz inequality,
    Jensen's inequality and the fact that the maximum of a set of positive numbers
    is smaller than the sum of all the elements of the set,
    we have for any $\lambda>0$
    $$\Gamma 
    \le
        \E\sup_{T\subset [p]: |T| = d} \|\Pi_T G\| 
    \le
    \frac 1 \lambda \log \sum_{T\subset [p]: |T| = d }\E \exp \lambda \|\Pi_T G\| 
    $$
    where for any support $T$, the matrix $\Pi_T$ is the orthogonal
    projection onto the subspace $\{\Sigma^{1/2}u, u\in\R^p \text{ s.t. } u_{T^c}=0\}$.
    The random variable $\|\Pi_T G\|^2$ has chi-square distribution
    with degrees of freedom at most $d$ and the function $G\to\|\Pi_T G\|$
    is 1-Lipschitz so that
    $\E \exp \lambda \|\Pi_T G\| 
    \le \exp(\lambda\sqrt{d} + \lambda^2/2)$ 
    by Theorem 5.5 in \cite{boucheron2013concentration}.
    There are $N={p \choose d}$ supports of size $d$
    so that 
    $$
    \Gamma\le
    \frac 1 \lambda\log\left(N \exp(\lambda\sqrt d + \lambda^2/2)\right)
    =
    \sqrt{d} + \frac \lambda 2 + \frac 1 \lambda \log N.
    $$
    Now set $\lambda = [2\log N]^{1/2}$ so that
    $\Gamma \le \sqrt{d} + \sqrt{2\log N}$.
    By a standard bound on binomial coefficients, $\log N\le d \log(ep/d)$
    and $\Gamma \le C_2 \sqrt{d\log(ep/d)}$
    for some absolute constant $C_2>0$.

    The second part of the proposition follows from
    \cite[Theorem 1.4 with $A=\X\Sigma^{-1/2}$]{plan_vershynin_liaw2017simple}
    and the upper bound on $\Gamma$ derived in the previous paragraph.
\end{proof}

\section{Preliminary probabilistic bounds}

\begin{proposition}
    \label{prop:Omega_2}
    Let $\beta^*\in\R^p$ with $T = \supp(\beta^*)$.
    Let $\lambda,\mu>0$ and define
    \begin{equation}
        \Rem(\beta^*,\lambda_0,\mu_0) \coloneqq \sqrt{
            \|\beta^*\|_0\left(\lambda_0^2 + 1\right) 
            + 
            \sigma^2
            (p-\|\beta^*\|_0)
            \E[(|Z|-\mu_0)_+^2]
        }.
        \label{def-Rem(lambda,mu)}
    \end{equation}
    Let $g_1,...,g_p$ be normal centered random variables
    (not necessarily independent) with $\E[g_j^2]\le \sigma^2$.
    Define the random variable $S(\lambda,\mu)$ by
    \begin{equation}
        S(\lambda,\mu)
        \coloneqq
        \sup_{u\ne 0}
        \frac{
            \sum_{j=1}^p g_j u_j
            +\lambda\sum_{j\in T}(|\beta^*_j| - |\beta^*_j + u_j|)
            -
            \mu \sum_{j\in T^c}|u_j|
            }{
            \|u\|
        }.
\end{equation}
Then 
$\E[ S(\lambda,\mu)]^2 \le
\sigma^2 \Rem(\beta^*,\lambda/\sigma,\mu/\sigma)^2$
where $Z\sim N(0,1)$.
\end{proposition}
\begin{proof}
    Let $s_j\in\{\pm 1\}$ be the sign of $\beta^*_j$ for $j\in T$.
    By simple algebra on each coordinate, we have almost surely
    \begin{equation}
        S(\lambda,\mu)
        \le
        \sup_{u\ne 0}\frac{\sum_{j\in T}(g_j-s_j\lambda)u_j +\sum_{j\in T^c} (|g_j|-\mu)_+|u_j| }{\|u\|}.
    \end{equation}
    Hence by the Cauchy-Schwarz inequality,
    $$S(\lambda,\mu)^2 \le\sum_{j\in T}(g_j-s_j\lambda)^2 + \sum_{j\in T^c} (|g_j|-\mu)_+^2.$$
    Thanks to the assumption on $g_1,...,g_p$, we 
    have $\E[\sum_{j\in T}(g_j-s_j\lambda)^2] \le \|\beta^*\|_0 (\lambda^2+\sigma^2)$
    as well as  the inequality
    $\E[(|g_j|-\mu)_+^2] \le \sigma^2 \E[(|Z|-\mu/\sigma)_+^2]$.
    Jensen's inequality yields $\E[S(\lambda,\mu)]\le\E[S(\lambda,\mu)^2]^{1/2}$ which completes the proof.
\end{proof}

\begin{proposition}[From a bound on the expectation to subgaussian tails]
    \label{prop:median-to-subg}
    Let $\beta^*\in\R^p$ and let $T$ be its support.
    Let $g_1,...g_p$ be centered jointly normal random variables with variance at most $\sigma^2$, let $\vg=(g_1,...,g_p)^\top$
    and $\bar\Sigma = \E[\vg\vg^\top]/\sigma^2$.
    Let $\lambda>\mu > 0$ and let $\theta_1>0$.
    Then for any $t>0$, with probability at least $1-e^{-t^2/2}$,
    \begin{align}
        \label{def-g-lipscthiz}
        &\sup_{u\in\R^p: u\ne 0}
        \left[
        \frac{\sum_{j=1}^p g_j u_j
            +\lambda\sum_{j\in T}(|\beta^*_j| - |\beta^*_j+u_j|)
            -
            {\mu}\sum_{j\in T^c}|u_j|)
            }{
            \max\left(\|u\|, \frac{1}{\theta_1} \|\bar\Sigma^{1/2} u\| \right)
        }
        \right]
        \\
        &\le
        \sigma \Rem(\beta^*,\lambda/\sigma,\mu/\sigma)
        + \sigma\theta_1 t
        .
    \end{align}
    where $\Rem(\cdot)$ is defined in \eqref{def-Rem(lambda,mu)}.
\end{proposition}
  
\begin{proof}
    Let $\vx\sim N(0,I_{p\times p})$ and write
    $\vg = \sigma \bar\Sigma^{1/2} \vx$.
    Denote by $G(\vx)$ the left-hand side of \eqref{def-g-lipscthiz}.
    Observe that $G$ is $(\sigma\theta_1)$-Lipschitz,
    so that by the Gaussian concentration theorem
    \cite[Theorem 5.6]{boucheron2013concentration},
    with probability at least $1-e^{-t^2/2}$ we have
    $G(\vx)
    \le
    \E[G(\vx)] + \sigma\theta_1 t
    $.
    \Cref{prop:Omega_2} provides an upper bound on $\E G(\vx)$.
\end{proof}

\section{Lasso upper bounds}
\label{appendix:overwhelming-proba}

Define the random variables 
\begin{equation}
    g_j \coloneqq n^{-1/2} \eps^\top \design e_j,
    \qquad 
    j=1,...,p,
    \qquad
    \text{ and set }
    \vg \coloneqq (g_1,...,g_p)^\top.
    \label{def-g_j}
\end{equation}
where $(e_1,...,e_p)$ is the canonical basis in $\R^p$.

\begin{theorem}
    \label{thm:overwhelming}
    There exists an absolute constant $C>1$ such that,
    if $p/k\ge C$ then the following holds.
    Let $T$ be the support of $\beta^*$.
    Let \Cref{assum:random-design} be fulfilled.
    Let $\lambda,t>0$
    and let $\Rem(\cdot)$ be defined by \eqref{def-Rem(lambda,mu)}.
    Define the constant $\theta_1$ and the sparsity level $s_1$ by
    \begin{equation}
        \theta_1 \triangleq \inf_{v\in\R^p: \|v_{T^c}\|_1\le\sqrt{s_1} \|v\|} \frac{\|\Sigma^{1/2} v\|}{\|v\|}
        \text{ where }
        \sqrt{s_1} = \inf_{\mu\in(0,\lambda)} \frac{\Rem(\beta^*,\lambda/\sigma,\mu/\sigma) + t}{(\lambda - \mu)/\sigma}
    \end{equation}
    and assume that $\theta_1 > 0$.
    Define $\rho_n(t) = 
    (t+1)/\sqrt n + (2/\theta_1)\sqrt{ (s_1/n) \log(p/s_1)})$ and assume that $\rho_n(t) < 1$.
    Then under \Cref{assum:random-design} the Lasso with tuning parameter
    $\lambda$ satisfies with probability at least $1-3\exp(-t^2/2)$ that
    \begin{align}
        &u\in \mathcal C(s_1)
        ,\qquad
        \Big|
        \frac{\|\X u\|}{\sqrt n \|\Sigma^{1/2}u\|} - 1 
        \Big|
        \le \rho_n(t),
        \\
        \text{ and }
        \quad
        &\theta_1 \|u\|
        \le
        \|\Sigma^{1/2} u\|
        \le
        \frac{\|\X u\|}{\sqrt n (1-\rho_n(t))}
        \le
        \frac{
            \sigma
            \{\Rem(\beta^*,\lambda/\sigma,\lambda/\sigma) +\theta_1 t \}
            }{
            \theta_1\sqrt n(1-\rho_n(t))^2
        }
        \label{RHS-C1}
    \end{align}
    all hold, where $u = \hbeta - \beta^*$
    and where $\mathcal C(s_1)$ is the cone
    \begin{equation}
        \label{eq:def-mathcal-C}
        \mathcal C(s_1)=\left\{
            u\in\R^p: \max\left(\|u\|,
            \frac{\|u_{T^c}\|_1}{\sqrt{s_1} } \right)
            \le \frac{\|\Sigma^{1/2} u\| }{ \theta_1 }
        \right\}.
    \end{equation}
\end{theorem}

To prove the result under \Cref{assum:random-design}, we will need the following which can be deduced
from the results of \cite{dirksen2015tail,plan_vershynin_liaw2017simple}.

\begin{proposition}
    \label{prop:quadratic-process}
    Let \Cref{assum:random-design} be fulfilled.
    Let $T\subset [p]$.
    Consider for some $s_1>0$ and some $\theta_1$ the cone
    \eqref{eq:def-mathcal-C}.
    Then with probability at least $1-e^{-t^2/2}$,
    \begin{equation}
        \label{eq:concentration-quadratic}
        \sup_{u\in \mathcal C(s_1): \Sigma^{1/2} u\ne 0}
        \frac{\Big| \frac 1 {\sqrt n} \|\X u\| - \|\Sigma^{1/2}u\|\Big|}{\|\Sigma^{1/2} u\| }
        \le
        \left(
        \frac{t+1}{\sqrt n}
        +
        r_n(s_1)
        \right)
    \end{equation}
    where $r_n(s_1)\coloneqq (2/\theta_1)\sqrt{(s_1/n)\log(p/s_1) }$.
\end{proposition}
\begin{proof}
    Let $\mathcal C=\mathcal C(s_1)$ for brevity.
    We apply Gordon's Theorem
    \cite[Corollary 1.2]{gordon1988milman} together with the Gaussian concentration theorem,
    which yields that for any $t>0$, with probability
    at least $1- 2e^{-t^2/2}$ we have
    \begin{equation}
        \sup_{u\in\mathcal C: \|\Sigma^{1/2} u\|=1}
        \Big| \|\X u\| - a_n\Big|
        \le
        t + \gamma(\Sigma^{1/2}\mathcal C)
    \end{equation}
    where
    $a_n = \E_{Z\sim N(0,I_{n\times n})}[\|Z\|]$ and
    $\gamma( \Sigma^{1/2} \mathcal C) \triangleq \E\sup_{u\in \mathcal C: \|\Sigma^{1/2} u\|= 1} G^\top\Sigma^{1/2} u$
    where the expectation is with respect to $G\sim N(0,I_{p\times p})$

    We now bound the expected supremum (i.e., the Gaussian mean width).
    Define $g_1,...,g_p$ by $g_j = G^\top\Sigma^{1/2} e_j$
    so that $g_j$ is centered normal with variance at most 1.
    Let $\mu>0$ that will be chosen later.
    For any $u\in \mathcal C$, write
    $
    G^T\Sigma u = (\sum_{j=1}^p g_j u_j - \mu\|u_{T^c}\|_1) + \mu\|u_{T^c}\|_1.
    $
    For the second term, for any $u\in C$ with $\|\Sigma^{1/2} u\|=1$ we have $\mu\|u_{T^c}\|_1\le \mu\sqrt{s_1}/\theta_1$.
    Next, to bound $\E\sup_{u\in\mathcal C:\|\Sigma^{1/2} u\|=1}(\sum_{j=1}^p g_j u_j - \mu\|u\|_1)$
    we use \Cref{prop:Omega_2} with $\lambda = 0$ and $\mu=\sqrt{2\log(p/s_1)}$
    by taking any $\beta^*$ with $\supp(\beta^*)=T$.
    This gives
    $$\gamma(\Sigma^{1/2}\mathcal C)
    \le (1/\theta_1)[\Rem(\beta^*,0,\mu)+ \mu\sqrt{s_1} ]
    \le
    (2/\theta_1) \mu\sqrt{s_1}.
    $$
    We complete the proof by observing that $|a_n-\sqrt n|\le 1$.
\end{proof}

\begin{proof}[Proof of \Cref{thm:overwhelming}
    ]
    Let $u=\hbeta-\beta^*$.
    By standard calculations similar to
    those used for
    \eqref{already_done}
    or, for instance,
    Lemma A.2 in \cite{bellec2016slope} we have
    \begin{equation}
        \label{initial2}
        \|\X u\|^2
        \le \eps^\top\X u + h(\beta^*) - h(\hbeta)
        =\sqrt n \sum_{j=1}^p g_j u_j + \sqrt n \lambda ( \|\beta^*\|_1 - \|u+\beta^*\|_1)
    \end{equation}
    where $g_1,...,g_p$ are defined in \eqref{def-g_j}.
    \Cref{prop:median-to-subg} with $\mu=\lambda$
    implies that
    with probability $1-e^{-t^2/2}$ we have
    \begin{equation}
        \|\X u\|^2
        \le
        \sqrt n
        N(u)
        \left[
        \Rem(\beta^*,\lambda, \lambda)
        + \sigma \theta_1 t
        \right]
        \label{beginning}
    \end{equation}
    where $N(u) = \|u\| \vee \frac{\|\Sigma^{1/2}  u\| }{\theta_1}$.
    Let $\mu$ be such that the infimum in the definition of $s_1$ is attained.
    We use again \Cref{prop:median-to-subg} with $\mu \in (0,\lambda)$,
    which implies that on an event of probability at least $1-e^{-t^2/2}$ we have
    \begin{equation}
        \|\X u\|^2
        \le
        \sqrt n
        \left(
        N(u)
        \left[
        \Rem(\beta^*,\lambda, \mu)
        + \sigma t
        \right]
        - (\lambda-\mu) \|u_{T^c}\|_1
        \right)
        \label{eq:FEJIWOFE}
    \end{equation}
    where we used that $\theta_1\le 1$ so that $\sigma\theta_1 t\le \sigma t$.
    We claim that on this event, $\theta_1\|u\|\le \|\Sigma^{1/2} u\|$ must hold.
    If this was not the case, then $N(u) = \|u\|$ and by the previous display we have  $\|u_{T^c}\|_1 \le \sqrt{s_1} \|u\|$,
    which implies $\theta_1\|u\|\le \|\Sigma^{1/2} u\|$ 
    by definition of $\theta_1$. 
    Hence 
    \begin{equation}
        N(u) = \frac{\|\Sigma^{1/2} u\|}{\theta_1}, 
        \quad
        \|u\|\le\frac{\|\Sigma^{1/2} u\|}{\theta_1}
\quad\text{ and }\quad
\frac{\|u_{T^c}\|_1}{\sqrt{s_1}}\le
\frac{\|\Sigma^{1/2} u\|}{\theta_1}
        \label{eq:F4}
    \end{equation}
    must hold.
    By the union bound,
    there is an event of probability at least
    $1-3e^{-t^2/2}$ on which \eqref{beginning},
    \eqref{eq:FEJIWOFE},
    \eqref{eq:F4}
    and \eqref{eq:concentration-quadratic} 
    must
    all
    hold. 
    In particular,
    \eqref{eq:F4} implies that  $u\in\mathcal C(s_1)$
    where $\mathcal C(s_1)$ is the cone \eqref{eq:def-mathcal-C}.
    By \eqref{eq:concentration-quadratic},
    inequality $\|\X u\|/\sqrt n \ge \|\Sigma^{1/2} u\| (1-\rho_n(t))$ holds for $u\in \mathcal C(s_1)$.
    Combining this with \eqref{beginning} completes the proof
    under \Cref{assum:random-design}.
\end{proof}

A useful bound on the quantity appearing in the statement
of \Cref{thm:overwhelming} is the following consequence of
\eqref{bound-1/K^3}:
\begin{equation}
\Rem(\beta^*, \lambda, \mu)^2
\le 
k (\lambda^2 + 1)
+ p \E[(|Z|-\mu)_+^2]
\le
k (\lambda^2 + 1)
+ p \frac{\exp(-\mu^2/2)}{\sqrt{2\pi}}
\frac{4}{\mu^3}
.
\label{bound-Rem_mu_lambda}
\end{equation}

\begin{proof}[Proof of \Cref{cor:random}]
    We apply \Cref{thm:overwhelming} to
    $\lambda = \sigma (1+\xi)\sqrt{2\log(ep/k)}$
    as in \Cref{cor:random}.
    We need to control
    $\Rem(\beta^*,\lambda/\sigma,\lambda/\sigma)$ as well as
    the sparsity $s_1$ defined 
    in \Cref{thm:overwhelming} for this choice of $\lambda$
    and a well-chosen $\mu \in (0,\lambda)$.
    Define 
    $\mu = \sigma\xi\sqrt{2\log(ep/k)}$,
    and without loss of generality, assume $\sigma^2=1$.
    Then, using \eqref{bound-Rem_mu_lambda},
    for $p/k$ large enough,
    $$
    \Rem(\beta^*,\lambda,\lambda)^2
    \le
    \Rem(\beta^*,\lambda,\mu)^2
    \le k(\lambda^2 + 1)
    + x k
    $$
    for any fixed $x\in(0,1)$ defined in the statement of \Cref{cor:random},
    so that
    $\Rem(\beta^*,\lambda,\lambda)
    \le \sqrt{k}\lambda + \sqrt k (1+x)$
    in the numerator of the right-hand side of \eqref{RHS-C1}.
    Furthermore, take $t=0.1 \lambda$ (which converges to $+\infty$ as
    $p/k\to +\infty$), so that the conclusions of \Cref{thm:overwhelming}
    hold with probability approaching one.
    Finally, notice that $\theta_1$ in \Cref{thm:overwhelming}
    satisfies $\theta_1^{-1}\le 1/C_{\min}$,
    and the sparsity level $s_1$ in \Cref{thm:overwhelming}
    satisfies
    $s_1\le C(x,\xi) k$, so if we set
    $\bar\rho_n = \rho_n(t)$ for $t=0.1\lambda$.
    \Cref{cor:random} follows.
\end{proof}

\begin{proof}[Proof of \eqref{critical:upper-bound}]
    We apply \Cref{thm:overwhelming} to
    \pb{$\lambda= \eqref{new_L}$}.
    We need to control
    $\Rem(\beta^*,\lambda/\sigma,\lambda/\sigma)$ as well as
    the sparsity $s_1$ defined 
    in \Cref{thm:overwhelming} for this choice of $\lambda$
    and a well-chosen $\mu \in (0,\lambda)$.
    Without loss of generality, by homogeneity we
    may assume $\sigma=1$.
    From \eqref{bound-Rem_mu_lambda},
    for $\mu=\lambda$ to bound from above \eqref{RHS-C1},
    with $\lambda$ equal to \eqref{new_L} we get
    $$
    \Rem(\beta^*, \lambda, \lambda)^2
    \le k(\lambda^2 + 1)
    + 
    \xi k \frac{\log(\frac{p}{\pb{k}})^{5/2}}{\lambda^3} 
    .
    $$
    If $\xi$ is fixed and $p/k\to +\infty$, then
    $\lambda/\sqrt{2\log(p/ k)} \to 1$ so that
    $\Rem(\beta^*, \lambda, \lambda)^2
    \le k (\lambda^2 + 1) + \xi k (1+o(1))$.
    We further take $t=0.1 \lambda$ (which converges to $+\infty$ as
    $p/k\to+\infty$), so that the right-hand side of
    \eqref{RHS-C1} is smaller than
    \begin{equation}
    \frac{\sqrt{k}(\lambda^2(1 + \xi (1+o(1))) + 1)^2 + 0.1\lambda}{\theta_1\sqrt n(1-\rho_n(t))^2}
    \le
    \frac{\lambda\sqrt k (1 + (1+o(1))\sqrt\xi) + 0.1\lambda +
        \sqrt k
    }{\theta_1\sqrt n(1-\rho_n(t))^2}
    \label{ewofjewio}
    \end{equation}
    thanks to $\sqrt{a+b+c}\le \sqrt{a} + \sqrt b + \sqrt c$
    for the last inequality.
    To derive \eqref{critical:upper-bound},
    we have $1/\theta_1 \le 1/C_{\min}$, and it remains to
    control $\rho_n(t)$ in the denominator.
    To control $\rho_n(t)$, we need to control the sparsity
    level $s_1$ defined in \Cref{thm:overwhelming}.
    Let $\mu = \lambda - 1/\lambda$
    (recall that we have assumed $\sigma=1$ without loss of generality).
    By definition of $s_1$ and $(a+b)^2\le 2(a^2+b^2)$ we have
    $$
    s_1
    \le 2 \Rem(\beta^*, \lambda, \mu)^2 \lambda^2 + 2t^2 \lambda^2
    $$
    and the second term is $2t^2\lambda=2\lambda^4 0.01$ since we have chosen $t=0.1\lambda$.
    Bounding from above the first term with \eqref{bound-Rem_mu_lambda}
    we get
    $$
    s_1
    \le 2 k(1+\lambda^2)\lambda^2
    + p \frac{\exp(-(\lambda - 1/\lambda)^2/2)}{\sqrt{2\pi}}\frac{4}{\mu^3}
    \lambda^2 + 2\lambda^4 0.01.
    $$
    Using
    $\exp(-(\lambda - 1/\lambda)^2/2)
    \le \exp(-\lambda^2/2) \exp(1)$ and the fact that
    $\lambda/\mu\to1$ as $p/(k\xi)\to +\infty$, for $p/k$ large enough we obtain
    $$
    s_1
    \le 2 k(1+\lambda^2)\lambda^2
    + 2 k \lambda^4 \xi
    + 2\lambda^4 0.01.
    $$
    Since $\lambda\asymp\sqrt{2\log(p/k)}$,
    this implies $s_1\lesssim k \log(p/k)^2$, and for $\rho_n(t)$
    for $t=0.1\lambda$ this gives up to multiplicative constants
    and using $\theta_1^{-1}\le C_{\min}^{-1}$ that
    $\rho_n(t) \lesssim \sqrt{\log(p/k)/n}
    +C_{\min}^{-1} \sqrt{\log(p/k)^3 k /n}$
    which converges to 0 provided that the assumption
    $\log(p/k)^3 k/n \to 0$ in \eqref{extra-asymptotic-regime-k} holds.
    This completes the proof of \eqref{critical:upper-bound},
    after noticing that the last two terms in
    \eqref{ewofjewio} are negligible and can be included
    inside the $(1+o(1))$ in \eqref{critical:upper-bound}.
\end{proof}

\section{Proof of Lasso lower bounds}
\label{appendix:proofs-lasso-lower}

Let us prove a general lower bound result that will imply
the lower bound \eqref{critical:lower-bound}.

\begin{theorem}
    \label{thm:lower-boud-klog(p/k)-new}
    \pb{Let $\xi>0$ be fixed.}
    There exists an absolute constant $C>0$ such that the following
    holds.
    Let \Cref{assum:random-design} be fulfilled
    and assume $\Sigma_{jj}=1$ for each $j=1,...,p$.
    Consider the asymptotic regime \eqref{extra-asymptotic-regime-k}.
    Then for $k,p,n$ large enough, we have for any $\beta^*$
    and any tuning parameter $\lambda\le \eqref{new_L}$ the lower bound
    \begin{equation}
        \label{lower-bound-first-part}
        \mathbb P\left(
        W
        \le [1+o(1)]\; \psi \; \risk
        \right) \ge 0.991,
    \end{equation}
    where $W$ is a random variable such that 
    \pb{
    $\E[W^2] \ge(1+o(1)) \xi k L(\frac{p}{k};\xi)$ with $L(\frac{p}{k};\xi)=\eqref{new_L}$
    }
    and where 
    \begin{equation}
        \label{def:psi}
    \psi=
    \sup_{u\in\R^p: \|u\| = 1, \|u\|_0 \le C k \log^2(p/k)} \| \Sigma^{1/2} u\|
    \end{equation}
    is the maximal sparse eigenvalue of order $Ck\log^2(p/k)$,
    and where $o(1)$ is a deterministic positive sequence that
    only depends on $(n,p,k)$ and converges
    to 0 in the asymptotic regime \eqref{extra-asymptotic-regime-k}.
    Furthermore, under the additional assumption that 
    $\|\Sigma\|_{op} = o(k\log(p/k))$ we have
    \begin{equation}
        \mathbb P\left(
        \sigma \sqrt{2 \xi k \log(p/k)}
        \le 
        [1+o(1)] \; \psi \; \risk
        \right) \ge 0.99.
    \end{equation}
\end{theorem}

\begin{proof}[Proof of \Cref{thm:lower-boud-klog(p/k)-new}]
    We will bound from below the noise barrier lower bound \eqref{nb-lower-bound}.
    By monotonicity of the noise barrier lower bound
    \eqref{nb-lower-bound} with respect to $\lambda$, it is enough to prove
    the result for $\lambda=\eqref{new_L}$;
    hence we assume that $\lambda=\eqref{new_L}$. 
    By scaling, we can also assume that $\sigma^2= 1$.
    Let $(g_j)_{j=1,...,p}$ be defined by \eqref{def-g_j} and
    define $u\in\R^p$ with $u_j = \text{sgn}(g_j)(|g_j| - \lambda)_+$,
    that is, $u$ is the soft-thresholding operator applied to $(g_1,...,g_p)$.
    The lower bound \eqref{nb-lower-bound} implies
    \begin{equation}
        W \coloneqq \left[\sum_{j=1}^p (|g_j| - \lambda)_+^2\right]^{1/2}
        =
        \frac{1}{\|u\|}
        \sum_{j=1}^ p
        g_j u_j - \lambda|u_j|
        \le
        \risk \frac{\|\X u\|}{\|u\|\sqrt n} .
        \label{def-Z}
    \end{equation}
    Since $\Sigma_{jj}=1$ for all $j=1,...,p$ and $\sigma^1=1$, each $g_j$ has $N(0,1)$
    distribution, hence
    $\E[W^2] = p \E[(|Z| - \lambda)_+^2]$.
    By \eqref{bound-1/K^3}, as $p/k\to+\infty$,
    $$
        p \E[(|Z| - \lambda)_+^2]
        \ge 
        (1+o_(1))
        \frac{4}{\sqrt{2\pi}} p \frac{\exp(-\lambda^2/2)}{\lambda^3}
        =
        \xi k \frac{\log(p/k)^{5/2}}{\lambda^3}
    $$
    thanks to $\lambda=\eqref{new_L}$.
    As $\xi$ is fixed and $p/k\to+\infty$,
    we have $\lambda^2\asymp \pb{2}\log(p/ k)$ so that,
    for $p/k$ large enough, it holds
    $$
        p \E[(|Z| - \lambda)_+^2]
        \ge 
        (1+o(1)) \xi \lambda^2 k.
    $$

    To complete the proof of the first part of the Theorem, it remains to bound
    $\frac{\|\X u\|}{\sqrt n \|u\|}$ from above in \eqref{def-Z}.
    We now prove that $\|u\|_0\le d$ with high probability,
    for some $d\ge 0$ that will be specified later.
    Since $u$ is the soft-thresholding operator applied to $(g_1,...,g_p)$,
    $\E[\|u\|_0] = p \mathbb P(|g|>L)= 2p Q(L)$ where $g\sim N(0,1)$
    and $Q$ is defined in \eqref{def-varpih-Phi}.
    By \eqref{bound-1/K^3} we have $\varphi(L)/L\asymp Q(L)$ as $L\to+\infty$,
    as well as $4Q(L)/L^2 \asymp \E[(|g|-L)_+^2]$.
    Hence for $p/k \ge C$ for some absolute constant $C>0$,
    and taking $L=\lambda=\eqref{new_L}$,
    we have
    $\E[\|u\|_0] \le 1.01 p \E[(|g|-L)_+^2] L^2 / 2 = 1.01 k L^4 /2$.
    By Markov's inequality, with probability at least 0.995 we have
    $\|u\|_0 \le 101 k L^4 \eqqcolon d$.

    Under the random design setting of \Cref{assum:random-design},
    $\Sigma$  is the covariance matrix
    of the rows of $\X$ and
    $\frac{\|\X u\|}{\sqrt n \|u\|} \le \psi \frac{\|\X u\|/\sqrt n}{\|\Sigma^{1/2} u\|}$.
    Furthermore, by \Cref{prop:cone-sparse-vectors-d}
    and the definition of $d$
    we have
    $$\mathbb P\left(
    \sup_{v\in\R^p: \|\Sigma^{1/2} v\| = 1, \|u\|_0 \le d} \left(\|\X v\|/\sqrt n\right) \le 1+\
    C_3\sqrt{
    \frac{ k\log^3(p/k) }{n} }
    \right) \ge 0.994
    $$
    for some absolute constant $C_3>0$.
    The quantity $k\log^3(p/k)/n$ converges to 0 in \eqref{extra-asymptotic-regime-k}. 
    Combining the two events with the union bound, we obtain
    \eqref{lower-bound-first-part}.

    We now prove the second part of the Theorem,
    under the additional assumption that $\|\Sigma\|_{op} = o( k\log(p/k))$.
    Since $\vg\sim N(0,\sigma^2\Sigma)$ (cf. \eqref{def-g_j}),
    let $\vx\sim N(0,I_p)$ be such that $\vg = \sigma \Sigma^{1/2}\vx$.
    Define $F(\vx) = W = \|u\|$, where, as above, $u$
    is the soft-thresholding operator applied to $\vg$.
    Then for two $\vx_1,\vx_2\in\R^{p}$, by the triangle inequality
    and the fact that the soft-thresholding operator is 1-Lipschitz,
    we have
    \begin{equation}
        |F(\vx_1) - F(\vx_2)|
        \le
        \| \Sigma^{1/2} (\vx_1 - \vx_2)^\top\| 
        .
    \end{equation}
    Hence $F$ is $\|\Sigma^{1/2}\|_{op}$-Lipschitz. 
    By the Poincar\'e inequality
    (see for instance \cite[Theorem 3.20]{boucheron2004concentration}),
    the variance of $F(\vx)$ is at most $\|\Sigma\|_{op}$, so that
    $\E[F(\vx)]^2 \ge \E[F(\vx)^2] - \|\Sigma\|_{op}$.
    By the Gaussian Concentration theorem \cite[Theorem 5.5]{boucheron2013concentration}, we have for any $t>0$
    $$\mathbb P\left( F(\vx) \ge \E[F(\vx)] - \|\Sigma\|_{op} t\right)\ge 1-e^{-t^2/2}.
    $$
    Since we have established above that $\E[F(\vx)^2] = \E[W^2]
    \asymp 2\sigma^2 k \log(p/k)$, this shows that
    $$\mathbb P\left[W\ge\sqrt{\pb{\xi} 2k\log(p/k)} [1-o(1)] \right]\ge 0.999$$
    by taking $t>0$ such that $0.001=e^{-t^2/2}$.
    The union bound and
    \eqref{lower-bound-first-part} complete the proof.
\end{proof}

\begin{proof}[Proof of \eqref{critical:lower-bound}]
    The above result implies \eqref{critical:lower-bound}
    since we simply have
    $\psi \le \|\Sigma\|^{1/2}\le C_{\max}$.
\end{proof}

\section{Proof of lower bounds for nuclear norm penalized estimators}
\label{appendix:proof-lower-bound-nuclear-norm}

\begin{proof}[Proof of \Cref{thm:nuclear-small-lambda}]
Inequality \eqref{nb-lower-bound} implies that
\begin{equation}
    \E\sup_{u\in\R^{m\times T}: \|\X u\|\le 1}\left[
        \sum_{i=1}^n \eps_i \Tr(X_i' u) - \sqrt n \lambda\|u\|_{S_1}
    \right]
    \le
    \E\risk
\end{equation}
For any $\gamma>0$,
\cite[Section 7, proof of Theorem 5]{rohde2011estimation} proves the existence of a finite set $\mathcal A^0$ of matrices such that $0\in \mathcal A^0$, the cardinality of $\mathcal A^0$ is at least $2^{r m /8}$ and
\begin{equation}
    \|u\|_{S_2} = \gamma \sqrt{r m/n},
    \qquad
    \text{rank }(u) \le r,
    \qquad
    \|u-v\|_{S_2} \ge (\gamma/8) \sqrt{r m/n}
\end{equation}
for any $u,v\in\mathcal A^0$.
Set $\gamma= 1/((1+\delta_r)\sqrt{r m})$ for the remaining of this proof.
Then for any $u\in\mathcal A^0$ we get
$\|\X u\|\le 1$ and
$\sqrt n\|u\|_{S_1}\le \sqrt{n r} \|u \|_{S_2}\le \sqrt r / (1+\delta_r)$. By restricting the above supremum to matrices in $\mathcal A^0$,
we get
\begin{equation}
    \E\sup_{u\in\mathcal A^0}\left[
        \sum_{i=1}^n \eps_i \Tr(X_i' u)
    \right]
    - \frac{\sqrt r \lambda }{ (1+\delta_r) }
    \le
    \E\risk.
\end{equation}
By Sudakov inequality 
(see for instance Theorem 13.4 in \cite{boucheron2013concentration}),
the properties of $\mathcal A^0$
and the Restricted Isometry property, we obtain a lower bound
on the expected supremum which yields the desired result.
\end{proof}

\end{document}